\documentclass[10pt, leqno]{article}

\usepackage{amsmath,amssymb,amsthm, epsfig}

\usepackage{amsmath}

\usepackage{amssymb}

\usepackage{hyperref}

\usepackage{color}

\usepackage{ulem}

\usepackage{epsfig}

\usepackage{dsfont}

\usepackage{mathrsfs}

\usepackage{mathtools}

\usepackage{stackrel}

\title{Lipschitz regularity of almost minimizers in a Bernoulli problem with non-standard growth}
\author{\it by \smallskip \\
 Jo\~{a}o Vitor da Silva,\footnote{\noindent \textsc{Jo\~ao Vitor da Silva}.
Universidade Estadual de Campinas - UNICAMP. Department of Mathematics. Campinas - SP, Brazil.
\texttt{E-mail address: jdasilva@unicamp.br}
}
\smallskip \qquad Anal\'ia Silva\footnote{\noindent \textsc{Anal\'ia Silva}.
 Departamento de Matem\'atica, FCFMyN, Universidad Nacional de San Luis and Instituto de Matem\'atica Aplicada San Luis (IMASL), UNSL - CONICET. Ejercito de los Andes 950, D5700HHW, San Luis, Argentina.
\texttt{E-mail address: acsilva@unsl.edu.ar}
}
\smallskip \\
\quad $\&$ \quad
\smallskip \\
Hern\'an Vivas\footnote{\noindent \textsc{Hern\'an Vivas}.
 Centro Marplatense de Investigaciones Matem\'aticas (CEMIM), Universidad Nacional de Mar del Plata, CIC, CONICET. De\'an Funes 3350, 7600, Mar del Plata, Argentina.
\texttt{E-mail address: havivas@unmdp.edu.ar}
}}

\newlength{\hchng}
\newlength{\vchng}
\def \N {\mathbb{N}}
\def \R {\mathbb{R}}

\newcommand{\defeq}{\mathrel{\mathop:}=}

\def\osc{\mathop{\text{\normalfont osc}}}
\def\diver{\mathop{\text{\normalfont div}}}

\newcommand{\ssubset}{\subset\joinrel\subset}

\newtheorem{theorem}{Theorem}[section]

\newtheorem{lemma}[theorem]{Lemma}

\newtheorem{proposition}[theorem]{Proposition}

\newtheorem{corollary}[theorem]{Corollary}

\theoremstyle{definition}

\newtheorem{definition}[theorem]{Definition}

\theoremstyle{remark}

\newtheorem{remark}[theorem]{Remark}

\numberwithin{equation}{section}

\newcommand{\intav}[1]{\mathchoice {\mathop{\vrule width 6pt height 3 pt depth  -2.5pt
\kern -8pt \intop}\nolimits_{\kern -6pt#1}} {\mathop{\vrule width
5pt height 3  pt depth -2.6pt \kern -6pt \intop}\nolimits_{#1}}
{\mathop{\vrule width 5pt height 3 pt depth -2.6pt \kern -6pt
\intop}\nolimits_{#1}} {\mathop{\vrule width 5pt height 3 pt depth
-2.6pt \kern -6pt \intop}\nolimits_{#1}}}

\begin{document}
\maketitle

\begin{abstract}

In this work we  establish the optimal Lipschitz regularity for non-negative almost minimizers of the one-phase Bernoulli-type  functional
$$
\mathcal{J}_{\mathrm{G}}(u,\Omega) \defeq \int_\Omega \left(\mathrm{G}(|\nabla u|)+\chi_{\{u>0\}}\right)\,dx
$$
where $\Omega \subset \mathbb{R}^n$ is a bounded domain and $\mathrm{G}: [0, \infty) \to [0, \infty) $ is a Young function with $\mathrm{G}^{\prime}=g$ satisfying the Lieberman's classical conditions. Moreover, of independent mathematical interest, we also address a H\"{o}der regularity characterization via Campanato-type estimates in the context of Orlicz modulars, which is new for such a class of non-standard growth functionals.

\bigskip


\bigskip


\end{abstract}



\section{Introduction and main result}

In this manuscript we deal with one-phase problems driven by the $\mathrm{G}-$energy functionals, with $\mathrm{G}$ a suitable function with non-polynomial growth. Specifically, we study nonnegative almost minimizers for one-phase Bernoulli-type functional as follows
$$
\mathcal{J}_{\mathrm{G}}(u,\Omega)\defeq \int_\Omega \left(\mathrm{G}(|\nabla u|)+\chi_{\{u>0\}}\right)\,dx
$$
where $\Omega \subset \mathbb{R}^n$ is a bounded Lipschitz domain and $\mathrm{G}: [0, \infty) \to [0, \infty) $ is a Young function with $\mathrm{G}^{\prime}=g$ satisfying
\begin{equation}\label{LiebCond}
\delta\leq \frac{tg^{\prime}(t)}{g(t)}\leq g_0 \qquad \text{for some constants} \qquad 1<\delta\leq g_0< \infty \quad (\text{Lieberman's condition}).
\end{equation}
In our research, almost minimizers mean that we will consider functions $0 \le u\in W^{1,\mathrm{G}}(\Omega)$ such that for some $\beta>0$ and constant $\kappa\geq0$ it holds that
\begin{equation}\label{VarIneqAlmMin}
  \mathcal{J}_{\mathrm{G}}(u,B_r(x))\leq (1+\kappa r^\beta)\mathcal{J}_{\mathrm{G}}(v,B_r(x))
\end{equation}
for any ball $B_r(x)$ such that $\overline{B_r(x)}\subset\Omega$ and any $v\in W^{1,\mathrm{G}}(B_r(x))$ such that $v=u$ on $\partial B_r(x)$ in the sense of traces (see Definition \ref{almostmin} and Subsection \ref{O-S-Subsection} for more details).

Particularly, in this framework, we address optimal local Lipschitz regularity for such class of almost minimizers, where the corresponding  estimate depends only on universal parameters
$$
\|\nabla u\|_{L^\infty(\Omega^{\prime})}\leq \mathrm{C}(\delta, g_0,\Omega^{\prime}, \beta, \kappa) \mathrm{G}^{-1}\left(1+\int_{\Omega}\mathrm{G}\left(|\nabla u|\right)\:dx\right) \quad \text{for any} \quad \Omega^{\prime} \subset \subset \Omega.
$$

We must highlight that the functional $\mathcal{J}_{\mathrm{G}}$ is a natural generalization of the classical one-phase (Bernoulli type) $p-$energy functional, which corresponds to the choice of $\mathrm{G}(t) = \frac{1}{p}t^p$ (for each $p \in \left(1, \infty\right)$ fixed). We will denote by $\mathcal{J}_p$ such functional.

Additionally, in the current literature, see \textit{e.g.} \cite{dipierro2022lipschitz}, almost minimizers have been interpreted as a sort of minimizers of ``perturbed functionals''. For instance, one can consider the functional with regional structure (which is quite similar to certain functionals arising in phase transitions for nonlinear nonlocal models - see \cite{CarGva21})
\begin{equation}\label{Convolut}
  \displaystyle \hat{\mathcal{J}}_{\mathrm{G}}(u,\Omega) = \mathcal{J}_{\mathrm{G}}(u,\Omega) + \frac{1}{2}\iint_{\Omega \times \Omega} \Phi(y, z)\Psi(u(y)-u(z))dydz,
\end{equation}
where $\Phi(y, z) = \rho(y)\rho(z)$ with  $\rho(x) = \Theta(u(x))$, and $\Theta: \mathbb{R} \to [0, 1]$ fulfils $\Theta \equiv 0$ for $x \leq 0$. Moreover, the potential $\Psi: \mathbb{R} \to [0, 1]$ satisfies $\Psi \equiv 0$ for $x \in (-\infty, 0]$

Now, if we define the $\Psi_u-$convolution as
$$
\displaystyle (\Phi \ast \Psi_u)(y, z) \defeq \iint_{\Omega \times \Omega} \Phi(y, z)\Psi(u(y)-u(z))dydz,
$$
then \eqref{Convolut} turns out to be
$$
  \displaystyle \hat{\mathcal{J}}_{\mathrm{G}}(u,\Omega) = \mathcal{J}_{\mathrm{G}}(u,\Omega) + \frac{1}{2}\Phi(y, z) \ast \Psi_u(y, z).
$$
Then, we verify that
$$
\mathcal{J}_{\mathrm{G}}(u,\Omega) \le \hat{\mathcal{J}}_{\mathrm{G}}(u,\Omega).
$$
On the other hand, it is easy to see
$$
\hat{\mathcal{J}}_{\mathrm{G}}(u,\Omega) \le \left(1 + \frac{1}{2}\omega_n r^n\right)\mathcal{J}_{\mathrm{G}}(u,\Omega) \le \left(1 + \frac{1}{2}\omega_n r^n\right)\mathcal{J}_{\mathrm{G}}(v,\Omega),
$$
for any $v\in W^{1,\mathrm{G}}(B_r(x))$ such that $v=u$ on $\partial B_r(x)$. In short, a minimizer for the ``perturbed functional'' $\hat{\mathcal{J}}_{\mathrm{G}}(u,\Omega)$ be revealed to be an almost minimizer for the $\mathrm{G}-$energy functional $\mathcal{J}_{\mathrm{G}}(u,\Omega)$ for appropriated choices of constants $\beta$ and $\kappa$. Therefore, it becomes a necessary task to work with almost minimizers instead of minimizers because of the impossibility of dealing with weak solutions of certain discontinuous functionals in certain general frameworks.

Regarding to minimizers in the linear setting, in Alt and Caffarelli's seminal paper \cite{AltCaf} concerning regularity minimizers to one-phase Bernoulli energy functional given by
$$
\displaystyle \mathcal{J}_2(u, \Omega) \defeq \int_{\Omega} \left(|\nabla u|^2 + \chi_{\{u>0\}}\right)dx
$$
it is established the locally Lipschitz regularity estimates. Moreover, it can be shown that minimizers $\displaystyle \mathcal{J}_2(u_0, \Omega) = \min_{\mathcal{K}} \mathcal{J}_2(u, \Omega)$ are solutions of the following free boundary problem
\begin{equation}\label{AltCaffProb}
  \left\{
\begin{array}{rclcl}
  \Delta u_0(x) & =& 0 & \text{in} & \{u_0>0\}\cap \Omega \\
  |\nabla u_0| & = & \sqrt{2} & \text{on} & \partial \{u_0>0\}\cap\Omega\\
  u(x) & =& g(x) & \text{on} & \partial \Omega.
\end{array}
\right.
\end{equation}
in an appropriate distributional sense.

The natural motivations to investigate such a class of free boundary problems of Bernoulli type  \eqref{AltCaffProb} comes from the analysis of cavities and jets type problems, see \textit{e.g.} \cite[Section 1.1]{CafSalBook}. Another relevant models also arise in combustion theory \cite{BCN90}, optimal design problems \cite{Tex10}, optimization problems with constrained volume \cite{AAC86}, shape optimization problems \cite{BucBut05}, \cite{BucGiac16} and \cite{BucVel15} and phase transitions \cite{PetroVald05}, \cite{PetVal05} and \cite{Vald06} just to mention a few scenarios.

Over the past few decades, there has been extensive literature exploring this research topic, particularly in the context of both single and two-phase problems, most notably the celebrated viscosity approach to the associated free boundary developed by Caffarelli in the series of trailblazing works \cite{Caff87}, \cite{Caff88} and \cite{Caff89}. We must refer the reader to \cite{DePSV21}, \cite{DeSFS19} and \cite{Velichkov23} for comprehensive modern essays on the subject.

Minimizers of the functional $\mathcal{J}_p$ have been also considered by Danielli and Petrosyan in \cite{DP05}, in which the regularity of the free boundary near flat points was addressed. Recently, Lipschitz estimates of the minimizers of the functional $\mathcal{J}_p$ has been addressed by DiPierro and Karakhanyan in \cite{DiPKar18}, where the authors also supplied the proof of the Lipschitz regularity when $p = 2$ without making use of monotonicity formulaes.

We must also recall Martinez-Wolanski's work \cite{MW08}, which consider the optimization problem of minimizing
$$
\mathcal{J}_{\mathrm{G}, \lambda}(u,\Omega) \defeq \int_\Omega \left(\mathrm{G}(|\nabla u|)+\lambda\chi_{\{u>0\}}\right)\,dx
$$
in the class $W^{1,\mathrm{G}}(\Omega)$ with $u-\phi_0\in W_0^{1,\mathrm{G}}(\Omega)$, for a bounded function $\phi_0\ge 0$ and $\lambda>0$. The authors prove that solutions to the optimization problem are locally Lipschitz continuous. Moreover, such solutions satisfy the corresponding free boundary problem of Bernoulli-type, thereby extending the Alt-Caffarelli's results for the scenario of Orlicz-Sobolev framework. Additionally, they address the Caffarelli’s classification scheme: flat and Lipschitz free boundaries are locally $C^{1,\alpha}$ for some $\alpha(\verb"universal") \in (0,  1)$. In \cite{BLO20} the authors extend these results to quasilinear singular/degenerate operators in the nonhomogeneous setting (i.e. with a nonzero right hand side).

Furthermore, viscosity approaches for one-phase problems driven by the $p-$Laplace and $p(x)-$Laplace operators has been developed in \cite{LeiRic18} and \cite{FerLed23} respectively. Additionally, in \cite{DaSRRV23}, the authors obtain the Lipschitz regularity of viscosity solutions of one-phase problems with non-homogeneous degeneracy (fully nonlinear operators with double phase signature) and some regularity properties of their free boundaries were addressed.

Now coming back to the setting of almost minimizers, we observe that since almost minimizers only fulfill a variational inequality, see \eqref{VarIneqAlmMin}, but not a proper PDE, the main obstacle in facing their regularity properties is the lack of a monotonicity formula as minimizers do (cf. \cite{AltCaf}). In this regard, it seems to be challenging to derive such techniques to the setting almost minimizers (cf. \cite{DET19} and \cite{DEST21} for related topics). In particular, in the quite recent article \cite{DeSS21}, De Silva and Savin developed a non-variational approach, based on Harnack-type inequality, for profiles that do not necessarily satisfy an infinitesimal equation.

We highlight that almost minimizers of $\mathcal{J}_2$ were widely investigated recently by several authors. In effect, we must cite the David  et al' work \cite{DET19}, where by developing an original approach combining techniques from potential theory and geometric measure theory, the authors obtain uniform rectifiability of the free boundary and, in the one-phase scenario, the corresponding $C^{1, \alpha}$ almost everywhere regularity. Therefore, the Alt-Caffarelli's classical results in \cite{AltCaf} were extended to the framework of almost minimizers.  We refer the reader to \cite{DEST21} for generalizations concerning variable coefficients. Moreover, the analysis of the semilinear scenario with variable coefficients
$$
\displaystyle \mathcal{F}_{\gamma}(v; \Omega) = \int_{\Omega} (\langle \mathbb{A}(x)\nabla v, \nabla v\rangle + q_{+}(v^{+})^{\gamma} + q_{-}(v^{-})^{\gamma})dx, \quad \text{for} \quad 0\le \gamma \le 1 \quad \text{and} \quad q_{\pm} \ge 0,
$$
where $\mathbb{A}$ is a matrix with H\"{o}lder continuous coefficients was explored in \cite{QueiTav18}, thereby proving sharp gradient estimates.

In \cite{DeSS20}, using non-variational techniques, De Silva and Savin provided a different approach from that of \cite{DET19} and \cite{DT15} to deal with almost minimizers of $\mathcal{J}_2$ and their free boundaries. Precisely, based on ideas developed by them in \cite{DeSS21}, they showed that almost minimizers of $\mathcal{J}_2$ are viscosity solutions in a more general sense. Once this was confirmed, the regularity of the free boundary for almost minimizers follows by applying the De Silva's techniques first developed in \cite{DeS11}.

Recently,  Dipierro \textit{et al} in \cite{dipierro2022lipschitz} obtained a Lipschitz continuity result to nonnegative almost minimizers for the nonlinear framework, namely for $\mathcal{J}_p$ in $B_1$, in the case $\max\left\{ 1, \frac{2n}{n+2}\right\}< p < \infty$. Precisely, there exists a universal constant $\mathrm{C}=\mathrm{C}(p, n, \beta, \kappa) >0$ such that
$$
\|\nabla u\|_{L^\infty(B_{1/2})}\leq \mathrm{C} \left(1+\|u\|_{W^{1,p}(B_1)}\right)
$$
Additionally, $u$ is uniformly Lipschitz continuous in a neighborhood of contact set $\{u = 0\}$. In effect, the authors' approach  was strongly inspired by the method introduced by De Silva and Savin in \cite{DeSS20}. In their setting, the main obstacles faced in are concerned the lack of linearity and the loss of exact descriptions of $p-$harmonic profiles in terms of mean value properties. To overcome such complication, the authors exploit some regularity estimates available in the classical literature.

For the vectorial scenarios we must quote the following contributions:

\begin{enumerate}
  \item For a singular system with free boundary, De Silva \textit{et al} in \cite{DeSSS22} study the regularity properties of vector-valued almost minimizers of the functional
$$
      \displaystyle \mathcal{J}_2({\textbf{u}};\mathrm{D}) = \int_{\mathrm{D}} \left(|\nabla \textbf{u}(x)|^2 + 2|\textbf{u}|\right)dx, 
$$ 
which is strongly related to a version of the classical obstacle problem (see \cite{Figalli2018}).

  \item For the weakly coupled vectorial $p-$Laplacian, given constant $\lambda > 0$ and a bounded Lipschitz domain $\mathrm{D} \subset  \mathbb{R}^n$ (for $n \ge  2$), Shahgholian \textit{et al} in \cite[Theorem 1.1]{Shahgholian2023} address local optimal Lipschitz estimates for almost minimizers of
$$
\displaystyle \mathcal{J}_p({\textbf{v}};\mathrm{D}) = \int_{\mathrm{D}} \sum_{i=1}^{m} |\nabla v_i(x)|^p + \lambda \chi_{\{|\bf{v}|>0\}}(x)dx, \quad (\text{for}\,\,\,1 < p < \infty),
$$
where ${\textbf{v}} = (v_1, \cdots , v_m)$, and $m \in \mathbb{N}$.

\end{enumerate}

Taking into account the above results, we will derive uniform Lipschitz estimates for a general class of almost minimizers of functionals with non-standard growth. Even if our approach is inspired by the works \cite{DeSS20} and \cite{dipierro2022lipschitz}, we highlight that the non-homogeneous nature of our functional entails several difficulties that add to the nonlinear character of the problem already present in the $p-$Laplace case. The lack of sharp embedding results and the hard task of handling  modulars and norms (for instance when applying H\"older's inequality) are two of these challenges which are successfully tackled in this manuscript, together with the general technical difficulty of dealing with general behavior different from a power and hence not necessarily homogeneous.

Another gap that we needed to fill was the absence of Campanato-type results dealing with modulars instead of norms. We consider this point to be of independent mathematical interest, as the results proved here (see Appendix) may be useful in other contexts that pertain regularity in Orlicz-Sobolev settings.

The main regularity result of this proposal is the following (see Section \ref{sec.prel} for details on the hypothesis):

\begin{theorem}\label{thm.main}
Let $\mathrm{G}: [0, \infty) \to [0, \infty)$ be a Young function such that $g=\mathrm{G}^{\prime}$ is convex and satisfies \eqref{LiebCond} with
\[
\frac{n(\delta+1)}{n-(\delta+1)}>g_0+1.
\]
Let $u\in W^{1,\mathrm{G}}(\Omega)$ be an almost minimizer of $\mathcal{J}_{\mathrm{G}}$ with exponent $\beta>0$ and constant $\kappa\geq0$. Then, $u$ is Lipschitz continuous in the interior of $\Omega$. Moreover, for any $\Omega^{\prime} \ssubset \Omega$ one has
\[
\|\nabla u\|_{L^\infty(\Omega^{\prime})}\leq \mathrm{C}\mathrm{G}^{-1}\left(1+\int_{\Omega}\mathrm{G}\left(|\nabla u|\right)\:dx\right) \quad \text{for any} \quad \Omega^{\prime} \subset \subset \Omega.
\]
where $\mathrm{C}=\mathrm{C}(\delta, g_0,\Omega^{\prime}, \beta, \kappa)>0$.
\end{theorem}

We stress that the nonlinear setting driven by $\mathcal{J}_{\mathrm{G}}(u,\Omega)$ provides further technical obstacles, e.g. the sum of two weak solutions is no longer a weak solution. Therefore, it plays an essential role to develop a theory of $g-$harmonic replacements for our nonlinear scenario to overcome such complications.

Furthermore, we point out that the regularity results in this work do not follow from the nowadays classical literature from calculus of variations for operators with Orlicz-Sobolev structure: Indeed, the integrand of the $\mathrm{G}-$energy functional has a discontinuous signature, instead, the non-standard growth scenarios that dealt with continuous integrand or even Lipschitz continuous one, see e.g. \cite{adams2003sobolev} and \cite{lieberman1991natural} for such related topic.

Concerning the strategy of the proof of Theorem \ref{thm.main}: a dichotomy result is the main step to establishing the optimal Lipschitz continuity. In effect, the dichotomy property ensures that either the average of the $\mathrm{G}-$energy of an almost minimizer to $\mathcal{J}_{\mathrm{G}}$ decreases in a suitable ball or the Luxemburg norm between its gradient and a  fixed vector becomes sufficiently small in a proper sense.

Summarizing, either the average of the $\mathrm{G}-$energy of an almost minimizer decreases in a smaller domain, or the almost minimizer is  close enough to a linear profile (with good a priori estimates). Once this dichotomy is proved, the next step is to show that one of such alternatives might be improved. Precisely, it will hold a kind of geometric improved of flatness, namely, if an almost minimizer is close to (in the sense of the average of the $\mathrm{G}-$energy) a suitable linear profile, in a ball, then it will be closer to a possibly other linear profile in a smaller ball. Therefore, by iterating such a reasoning, together with the dichotomy property, provides the desired Lipschitz regularity of almost minimizers to $\mathcal{J}_{\mathrm{G}}$,

Finally, in our scenario, we need to show that the sum of $\mathrm{G}-$harmonic replacements solve an equation for which $C^{1, \alpha}$ estimates hold. In effect, this holds true using a linearization strategy, together with the assumptions of the dichotomy. The main difficulty appears in the proof of the improvement of dichotomy’s second alternative. The original strategy to solve such an obstacle  strongly relies on De Silva-Savin's work \cite{DeSS20}, where the authors often use harmonic replacements as competitors. This property holds, for harmonic replacements allowing achieve estimates for the average of their energy, exploiting Schauder estimates instead of gradient estimates.

\section*{Some extensions and further comments}\label{Extensions}

We will present a number of interesting building-block models in which our results work as well.

\begin{enumerate}

  \item {\bf Models arising of non-autonomous functionals:}

First, by considering the quantitative results (e.g. regularity estimates, Harnack inequality, see references below) we can to deal with the following model cases:

\begin{itemize}

  \item {\bf Problems with $(p\& q)-$growth.}
$$
 u_0+W_0^{1,p}(\Omega) \ni w \mapsto \int_{\Omega} \left(\mathcal{H}_0(x,\nabla w)+\chi_{\{w>0\}}\right)dx,
$$
as long as the integrand $\mathcal{H}_0$ enjoys an appropriate double phase type structure
$$
   L_1 \cdot \left(\frac{1}{p}|\xi|^p+\frac{1}{q}|\xi|^q\right)\leq \mathcal{H}_0(x,\xi)\leq L_2 \cdot \left(\frac{1}{p}|\xi|^p+\frac{1}{q}|\xi|^q\right),
$$
where $0<L_1 \le L_2< \infty$ and   $1<p<q<\infty$, which make possible to access available existence/regularity results for weak solutions. We recommend to reader see the series of  fundamental manuscripts \cite{BCM15}, \cite{CM15I}, \cite{DeFM19} , \cite{DeFM20} and  \cite{DeFO19} for related topics.

  \item {\bf Problems with Uhlenbeck's type structure.}

    We also stress that our approach is particularly refined and quite far-reaching in order to be employed in other classes of problems. Indeed,

	We can also extended our results for Musielak–Orlicz Spaces with non-homogeneous term as follow (see \cite[\S 2.3]{DHHR11})
\[
            \displaystyle u_0+W^{1, \varphi}(\Omega)\ni v \mapsto \int_{\Omega} \left(\varphi(x, |\nabla v|) +\chi_{\{v>0\}}\right)dx
\]
     for $(x, \xi) \mapsto \varphi(x, |\xi|)$ with particular structure, including:

\begin{enumerate}
  \item[\checkmark] $p-$growth, i.e., $\varphi(x,|t|) = \frac{1}{p}|t|^p$;
  \item[\checkmark] Orlicz growth, i.e., $\varphi(x,|t|) = \mathrm{G}(|t|)$;
  \item[\checkmark] Multi phase growth, i.e., $\displaystyle \varphi(x,|t|) = \frac{1}{p}|t|^p + \sum_{i=1}^{k} \mathfrak{a}_i(x)\frac{1}{q_i}|t|^q$ for $\le \mathfrak{a}_i \in C^{0, 1}(\Omega)$;
\end{enumerate}
just to mention a few, (see \cite{Chl18} and \cite[Section 2]{HO19}).

In this case, $\varphi$ must satisfy the Uhlenbeck's structure: for constants $1< p\le q < \infty$
$$
 t \mapsto \varphi(x, t) \quad \text{is in } \quad C^1([0, \infty)) \cap C^2((0, \infty)) \quad \text{and} \quad p-1 \leq \frac{t \varphi^{\prime \prime}(x, t)}{\varphi^{\prime}(x, t)} \leq \max_{1\le i \le k}\{q_i-1\}.
$$
\end{itemize}

\item {\bf Models with generalized doubly degenerate growth.}

Finally, we may consider a wide class of nonlinear degenerate models with non standard growth properties (cf. \cite{BO20}, and \cite{Chl18} for a survey). An archetypical example we have in mind concerns models of the form
$$
\mathcal{J}_{\mathcal{G}_{\mathrm{G}, \mathrm{H}}}(u,\Omega) \defeq  \int_{\Omega} \left(\mathcal{G}_{\mathrm{G}, \mathrm{H}}(x, |\nabla u|) +\chi_{\{u>0\}}\right)dx
$$
where
$$
\mathcal{G}_{\mathrm{G}, \mathrm{H}}(x, \xi) \defeq \mathrm{G}(|\xi|)+\mathfrak{a}(x)\mathrm{H}(|\xi|) \quad \text{with}\quad 0\le\mathfrak{a} \in C^{0, 1}(\Omega),
 $$
for Young functions $\mathrm{G}, \mathrm{H}: [0, \infty) \to [0, \infty)$ with $\mathrm{G}, \mathrm{H} \in \Delta_2 \cap \nabla_2$, $\mathrm{G} \prec \mathrm{H} \prec \mathrm{G}^{1+\frac{1}{n}}$, and satisfying the Lieberman's condition: there exist constants $\mathrm{c}_{\mathrm{G}}, \mathrm{c}^{\prime}_{\mathrm{G}}, \mathrm{c}_{\mathrm{H}}, \mathrm{c}^{\prime}_{\mathrm{H}} \geq 1$ such that
$$
\mathrm{c}^{\prime}_{\mathrm{G}} \leq \frac{t\mathrm{G}^{\prime}(t)}{\mathrm{G}(t)} \leq \mathrm{c}_{\mathrm{G}} \quad \text{and} \quad \mathrm{c}^{\prime}_{\mathrm{H}} \leq \frac{t\mathrm{H}^{\prime}(t)}{\mathrm{H}(t)} \leq \mathrm{c}_{\mathrm{H}} \quad \text{for all} \quad  t>0.
$$

\end{enumerate}

\bigskip

The rest of the paper is organized as follows: in Section \ref{sec.prel} we present the preliminary definitions and results that will be needed throughout the paper; Section \ref{sec.dichot} is devoted to the proof of the most important technical results of the paper, namely Lemmas \ref{lem.dichotomy} and \ref{lemalargo} and, as a consequence, the crucial interior gradient estimate in Corollary \ref{cor}. In Section \ref{sec.main} we prove the main result of this paper, i.e. Theorem \ref{thm.main}. Finally, the Appendix contains the proof of the Campanato-type regularity result that is used in Corollary \ref{cor}, summarized in Theorem \ref{campanato}.

\section{Preliminaries}\label{sec.prel}

In this section we present some preliminary definitions and results that will be used in the proofs of the main results of the manuscript. We start with some basic facts regarding Young functions and Orlicz-Sobolev spaces; these are rather well known but we include them for the sake of completeness.

\subsection{Young functions}
An application $\mathrm{G}\colon[0,\infty)\longrightarrow [0,\infty)$ is said to be a  \emph{Young function} if it admits the integral representation $\displaystyle \mathrm{G}(t)=\int_0^t g(\tau)\,d\tau$, where $g$ is a right continuous function defined on $[0,\infty)$ and satisfying:

$$
\left\{
\begin{array}{l}
g(0)=0, \quad g(t)>0 \text{ for } t>0,   \\
g \text{ is nondecreasing on } (0,\infty), \\
\displaystyle \lim_{t\to\infty}g(t)=\infty.
\end{array}
\right.
$$

From these properties it is easy to see that a Young function $\mathrm{G}$ is continuous, nonnegative, strictly increasing and convex on $[0,\infty)$. Further, we recall that we may extend $g$ to the whole $\R$ in an odd fashion: for $t<0$ $g(t)=-g(-t)$. And we may (and will) assume without loss of generality that $\mathrm{G}(1)=\mathrm{G}^{-1}(1)=1$. Note that $\mathrm{G}$ is strictly increasing and hence $\mathrm{G}^{-1}$ is well defined.

In this work we will consider the class of Young functions such that $g=\mathrm{G}^{\prime}$ is an absolutely continuous function that satisfies the condition
\begin{equation}\label{eq.lieberman}
\delta\leq \frac{tg^{\prime}(t)}{g(t)}\leq g_0
\end{equation}
for some constants $1<\delta\leq g_0$. This condition was first considered in the seminal work of G. Lieberman \cite{lieberman1991natural} and is the analogous to the ellipticity condition in the linear theory. For our purposes, we will make the (rather standard assumption) that
\begin{equation}\label{eq.exponentes}
(\delta+1)^\ast>g_0+1
\end{equation}
where $(\delta+1)^\ast$ is the Sobolev conjugate of $(\delta+1)$, i.e. $(\delta+1)^\ast=\frac{n(\delta+1)}{n-(\delta+1)}$. We will also assume that
\[
g\text{ is convex.}
\]
This last assumption is analogous to the degenerate setting in the homogeneous problem, i.e. the case where $p\geq 2$.

Equation \eqref{eq.lieberman} implies that that $\mathrm{G}$ satisfies
\begin{equation}\label{eq.delta2}
\delta+1\leq \frac{tg(t)}{\mathrm{G}(t)} \leq g_0+1;
\end{equation}
this condition is equivalent to ask $\mathrm{G}$ and $\tilde{\mathrm{G}}$ to satisfy the \emph{$\Delta_2$ condition or doubling condition}, i.e.,
\begin{equation}\label{Delta_2cond}
\mathrm{G}(2t)\leq 2^{g_0+1} \mathrm{G}(t), \qquad \tilde{\mathrm{G}}(2t) \leq 2^{1+\frac{1}{\delta}} \tilde{ \mathrm{G}}(t),
\end{equation}
where the \emph{complementary} function of a Young function $\mathrm{G}$ is the Young function $\tilde{\mathrm{G}}$ defined as
$$
\tilde{\mathrm{G}}(t)=\sup\{ta-\mathrm{G}(a)\colon a>0\}.
$$
The doubling condition allow us to split sums as
\begin{equation}\label{eq.sums}
\mathrm{G}(a+b)\leq 2^{g_0} (\mathrm{G}(a)+\mathrm{G}(b)).
\end{equation}
Moreover, integrating \eqref{eq.delta2} we have that
\begin{align}
t^{g_0+1}\leq \mathrm{G}(t) \leq t^{\delta+1},\quad\text{if }0<t\leq 1 \label{eq.delta2less1}\\
t^{\delta+1}\leq \mathrm{G}(t) \leq t^{g_0+1},\quad\text{if }t>1. \nonumber
\end{align}

As mentioned in the Introduction, the lack of homogeneity is a serious technical difficulty in our context. The following lemma gives the alternatives that hold in this scenario; its proof is elementary so we omit it.
\begin{lemma}
For $\theta\in[0,1]$ and  $t\geq 0$
$$
\mathrm{G}(\theta t) \leq \theta \mathrm{G}(t),
$$
and for $\theta\geq 1$ and  $t\geq 0$
$$
\mathrm{G}(\theta t)\geq \theta \mathrm{G}(t).
$$
More generally,  for any, $\theta,t\geq 0$
\begin{align}
\mathrm{G}(t)\min\left\{\theta^{\delta+1}, \theta^{g_0+1}\right\} \leq \mathrm{G}(\theta t) \leq \mathrm{G}(t)\max\left\{\theta^{\delta+1}, \theta^{g_0+1}\right\},  \label{minmax2} \\
\mathrm{G}^{-1}(t)\min\left\{\theta^\frac{1}{\delta+1}, \theta^\frac{1}{g_0+1}\right\} \leq \mathrm{G}^{-1}(\theta t) \leq \mathrm{G}^{-1}(t)\max\left\{\theta^\frac{1}{\delta+1}, \theta^\frac{1}{g_0+1}\right\}.\label{minmax3}
\end{align}
\end{lemma}

\subsection{Orlicz-Sobolev spaces}\label{O-S-Subsection}

Given a Young function $\mathrm{G}$ and a bounded open set $\Omega\subset \mathbb{R}^n$, we consider the spaces $L^{\mathrm{G}}(\Omega)$ and $W^{1,\mathrm{G}}(\Omega)$ defined as follows:
$$
\begin{array}{rcl}
  L^{\mathrm{G}}(\Omega) & = & \{u : \mathbb{R} \to \mathbb{R} \quad \text{measurable such that} \quad \Phi_{\mathrm{G},\Omega}(u) < \infty\} \\
  W^{1,\mathrm{G}}(\Omega) & = & \{u \in L^{\mathrm{G}}(\Omega) \quad \text{such that} \quad  \Phi_{\mathrm{G},\Omega}(|\nabla u|) < \infty\}
\end{array}
$$
where the modular $\Phi_{\mathrm{G},\Omega}(u)$ stands for
$$
\displaystyle \Phi_{\mathrm{G},\Omega}(u) = \int_{\Omega} \mathrm{G}(|u|) dx.
$$
These spaces are endowed with the so-called Luxemburg norm defined as follows:
$$
\displaystyle \|u\|_{L^{\mathrm{G}}(\Omega)} = \inf\left\{ \lambda > 0: \,\,\, \Phi_{\mathrm{G},\Omega}\left(\frac{u}{\lambda}\right) \le 1
\right\}
$$
and
$$
\|u\|_{W^{1,\mathrm{G}}(\Omega)} = \|u\|_{L^{\mathrm{G}}(\Omega)}+\|\nabla u\|_{L^{\mathrm{G}}(\Omega)}
$$
In general, it is easy to see that
\begin{equation}\label{eq.normandmoular}
\|w\|_{L^{\mathrm{G}}(\Omega)}\leq\max\left\{\Phi_{\mathrm{G},\Omega}(w)^\frac{1}{g_0+1},\Phi_{\mathrm{G},\Omega}(w)^\frac{1}{\delta+1}\right\}.
\end{equation}

\subsection{Technical results concerning $g-$harmonic functions}

The notion of $g-$harmonic function will of course be of great importance in our work; In effect, $g-$harmonic functions are weak solutions of the equation $\Delta_gv=0$ where $\Delta_g$ is the $g-$Laplacian operator given by
\[
\Delta_gv=\diver\left(g(|\nabla v|)\frac{\nabla v}{|\nabla v|}\right).
\]
More precisely, we will be interested in unique weak solutions of the Dirichlet problem
\begin{equation}\label{eq.gharm}
\left\{
\begin{array}{rcrcl}
  \Delta_g v & = & 0 & \text{in} & \Omega \\
  v & = & h & \text{on} & \partial \Omega
\end{array}
\right.
\end{equation}
Such weak solutions can be constructed as minimizers of the energy functional $\displaystyle u\rightarrow\int_\Omega \mathrm{G}(|\nabla u|)\:dx$ (with fixed Dirichlet boundary condition) via the Direct Method of the Calculus of Variations.

The following proposition will be used, its proof is contained in the Lieberman's fundamental work \cite[Section 5]{lieberman1991natural}. Before stating it, we recall that a constant will be called universal if it depends only on $n,\delta$ and $g_0$ (and eventually $\beta$ or $\kappa$ in Definition \ref{almostmin}); multiplicative constants will be denoted by $C$ and in the proofs they may change from line to line.

\begin{proposition}\label{prop.reg}
Let $v\in W^{1,\mathrm{G}}(B_R)$ be a $g-$harmonic function. Then,
\begin{enumerate}

\item there exists a universal constant $\mathrm{C}>0$ such that for any $x\in B_{R/2}$
\begin{equation}\label{eq.supavg}
\sup_{B_{R/4}(x)}\mathrm{G}\left(|\nabla v|\right)\leq \mathrm{C}\intav{B_{R/2}(x)}\mathrm{G}\left(|\nabla v|\right)\:dx.
\end{equation}

\item there exist universal constants $\mathrm{C}>0$ and $\alpha\in(0,1)$ such that
\begin{equation}\label{eq.osc}
\osc_{B_\rho}|\nabla v|\leq \mathrm{C}  \sup_{B_R}|\nabla v|\left(\frac{\rho}{R}\right)^\alpha,\quad 0<\rho<R.
\end{equation}
\end{enumerate}
\end{proposition}

\begin{remark} \label{remark.reg}
The estimates in Proposition \eqref{prop.reg} still hold under the more general assumption that $v$ satisfies $\diver(\mathbb{A}(x)\nabla v)=0$ with $x \mapsto \mathbb{A}(x)$ satisfying the structural condition:
\[
\lambda\frac{g(|z|)}{|z|}|\xi|^2\leq \mathbb{A}(x)\xi\cdot\xi\leq \Lambda\frac{g(|z|)}{|z|}|\xi|^2.
\]
See \cite{lieberman1991natural} for details.
\end{remark}

Another important concept that will be used is that of harmonic replacement in a ball, which roughly speaking is the $g-$harmonic that coincides with a given function on the boundary of a ball:
\begin{definition}
The $g-$harmonic replacement of a function $u\in W^{1,G}(\Omega)$ in $B_r(x)$ is the unique $g-$harmonic function $v\in W^{1,G}(B_r(x))$ that coincides with $u$ on $\partial B_r(x)$ in the sense of traces, i.e. the unique weak solution of \eqref{eq.gharm} with $\Omega=B_r(x)$ and $h=u$.
\end{definition}

The following technical lemmas will be used later: the first one is a bound on the energy of difference between and function and its $g-$harmonic replacement and the other one is needed to ensure the desired ellipticity (in the sense of Remark \ref{remark.reg}) so that interior estimates apply:
\begin{lemma}\label{lem.harmrep}
Let $G$ a Young function satisfying \eqref{eq.lieberman} and such that $\mathrm{G}^{\prime}=g$ is a convex function, $u\in W^{1,G}(B_r)$ and $v$ the $g-$harmonic replacement of $u$ in $B_r$. Then the following inequality holds for a universal constant $\mathrm{C}>0$
$$
\int_{B_r}\mathrm{G}(|\nabla u-\nabla v|)\,dx\leq \mathrm{C}\int_{B_r}\mathrm{G}(|\nabla u|)-\mathrm{G}(|\nabla v|)\, dx.
$$
\end{lemma}

\begin{proof}
As $v$ is a weak solution of $\Delta_g u=0$, we know that
\begin{equation}\label{weak}
\int_{B_r} g(|\nabla v|)\frac{\nabla v}{|\nabla v|}\nabla\varphi \,dx=0,
\end{equation}
for all $\varphi\in W_0^{1,\mathrm{G}}(B_r)$.

We consider the following convex combination
$$
u^s(x)=su(x)+(1-s)v(x)\qquad 0\leq s\leq1
$$
Observe that $u^0=v$ and $u^1=u$. Now using that $u-v\in W^{1,G}_0(\Omega)$ and \eqref{weak}, we have
\begin{align*}
\int_{B_r}\mathrm{G}(|\nabla u|)-\mathrm{G}(|\nabla v|)\, dx &=\int_{B_r}\mathrm{G}(|\nabla u^1|)-\mathrm{G}(|\nabla u^0|)\, dx\\
&=\int_{B_r}\left(\int_0^1 \frac{d}{ds}\mathrm{G}(|\nabla u^s|)\,ds\right)\,dx\\
&=\int_{B_r}\left(\int_0^1 g(|\nabla u^s|)\frac{\nabla u^s}{|\nabla u^s|}\cdot\nabla (u-v)\,ds\right)\,dx\\
&=\int_{B_r}\int_0^1 \left(g(|\nabla u^s|)\frac{\nabla u^s}{|\nabla u^s|}-g(|\nabla v|)\frac{\nabla v}{|\nabla v|}\right)\cdot\nabla (u-v)\,ds\,dx.
\end{align*}
Next, note that $u^s(x)-v(x)=su(x)-sv(x)=s(u(x)-v(x))$; then using that
\[
\left(g(|a|)\frac{a}{|a|}-g(|b|)\frac{b}{|b|}\right)\cdot\nabla (a-b)\geq \mathrm{C} \mathrm{G}(|a-b|)
\]
for a constant $\mathrm{C}=\mathrm{C}(\delta)>0$ (see \cite[Lemma 3.1]{cantizano2021continuity} for a proof) and \eqref{minmax2} the last above term is bounded in the following way
{\small{
$$
\begin{array}{rcl}
  \displaystyle \int_{B_r}\frac{1}{s}\int_0^1 \left(g(|\nabla u^s|)\frac{\nabla u^s}{|\nabla u^s|}-g(|\nabla v|)\frac{\nabla v}{|\nabla v|}\right)\cdot\nabla (u^s-v)\,ds\,dx & \ge  & \displaystyle  \mathrm{C}\int_0^1\frac{1}{s}\int_{B_r} \mathrm{G}(|\nabla u^s-\nabla v|)\,dx\,ds\\
   & \ge & \displaystyle  \mathrm{C}\int_0^1\frac{s^{g_0+1}}{s}\int_{B_r} \mathrm{G}(|\nabla u-\nabla v|)\,dx\,ds\\
   & \ge & \displaystyle  \frac{\mathrm{C}}{g_0+1}\int_{B_r} \mathrm{G}(|\nabla u-\nabla v|)\,dx
\end{array}
$$}}
and the desired inequality follows.
\end{proof}

\begin{lemma}\label{lem.ellipticity}
Let $q\in\R^n$ and $|h(x)|<\frac{|q|}{2}$. If $\overrightarrow{\mathbf{F}}(z)=g(|z|)\frac{z}{|z|}$ for a Young function satisfying \eqref{eq.lieberman} and
\[
\mathfrak{A}(x) \defeq \int_0^1D\overrightarrow{\mathbf{F}}\left(q+th(x)\right)\:dt
\]
then
\begin{equation}\label{eq.Aellip}
\lambda\frac{g(|q|)}{|q|}|\xi|^2\leq \mathfrak{A}(x)\xi\cdot\xi\leq \Lambda\frac{g(|q|)}{|q|}|\xi|^2
\end{equation}
for any $\xi\in \R^n\setminus\{0\}$ and for some constants $0<\lambda\leq\Lambda$ depending only on $\delta$ and $g_0$.
\end{lemma}

\begin{proof}

We start by computing $D\overrightarrow{\mathbf{F}}$: if $i\neq j$
\[
D_i\overrightarrow{\mathbf{F}}_j(z) = D_i\left(g(|z|)\frac{z_j}{|z|}\right) = g^{\prime}(|z|)\frac{z_iz_j}{|z|^2}-g(|z|)\frac{z_iz_j}{|z|^{3}}
\]
while
\[
D_i\overrightarrow{\mathbf{F}}_i(z) = D_i\left(g(|z|)\frac{z_i}{|z|}\right)= g^{\prime}(|z|)\frac{z_i^2}{|z|^2}+g(|z|)\left(\frac{1}{|z|}-\frac{z_i^2}{|z|^3}\right).
\]
Then,
{\small{
\begin{align*}
D\overrightarrow{\mathbf{F}}(z)\xi\cdot\xi & = \sum_{i\neq j}\left(g^{\prime}(|z|)\frac{z_iz_j}{|z|^2}-g(|z|)\frac{z_iz_j}{|z|^{3}}\right)\xi_i\xi_j
+\sum_i\left(g^{\prime}(|z|)\frac{z_i^2}{|z|^2}+g(|z|)\left(\frac{1}{|z|}-\frac{z_i^2}{|z|^{3}}\right)\right)\xi_i^2 \\
 				 & = \sum_{i\neq j}\left(g^{\prime}(|z|)-\frac{g(|z|)}{|z|}\right)\frac{z_iz_j\xi_i\xi_j}{|z|^2}+g^{\prime}(|z|)|\xi|^2  \\
 				 & = \left(g^{\prime}(|z|)-\frac{g(|z|)}{|z|}\right)\frac{1}{|z|^2}\left(z\cdot\xi\right)^2+\frac{g(|z|)}{|z|}|\xi|^2  \\
 				 & \geq  \left(\frac{g(|z|)}{|z|}-\left(g^{\prime}(|z|)-\frac{g(|z|)}{|z|}\right)\right)|\xi|^2  \\
 				 & \geq  g_0\frac{g(|z|)}{|z|}|\xi|^2 \\
                 & \geq g^{\prime}(|z|)|\xi|^2.
\end{align*}}}
This implies
\[
\mathfrak{A}(x)\xi\cdot\xi \geq \int_0^1 g^{\prime}(|q+th(x)|)\:dt|\xi|^2\geq g^{\prime}\left(\frac{|q|}{2}\right)|\xi|^2\geq  \delta\frac{g\left(\frac{|q|}{2}\right)}{\frac{|q|}{2}}|\xi|^2\geq \lambda \frac{g\left(|q|\right)}{|q|}|\xi|^2
\]
where $\lambda \defeq 2^{1-g_0}\delta$, and we have used that $|q+th(x)|\geq \frac{|q|}{2}$ by hypothesis, that $g^{\prime}$ is increasing and the doubling condition for $g$. This gives the lower bound on \eqref{eq.Aellip}.

Similarly, we find that
\[
D\overrightarrow{\mathbf{F}}(z)\xi\cdot\xi \leq  \left(\frac{g(|z|)}{|z|}+\left(g^{\prime}(|z|)-\frac{g(|z|)}{|z|}\right)\right)|\xi|^2=g^{\prime}(|z|)|\xi|^2
\]
and
\[
\mathfrak{A}(x)\xi\cdot\xi \leq g^{\prime}\left(\frac{3|q|}{2}\right)|\xi|^2\leq g_0\frac{g\left(\frac{3|q|}{2}\right)}{\frac{3|q|}{2}}|\xi|^2\leq \Lambda \frac{g\left(|q|\right)}{|q|}|\xi|^2
\]
now with $\Lambda \defeq \left(\frac{3}{2}\right)^{g_0-1}g_0$, this time using that $|q+th(x)|\leq \frac{3|q|}{2}$.
\end{proof}

\subsection{Almost minimizers}

In this section we give the rigorous definition of an almost minimizer and show that they satisfy a nice scaling property:

\begin{definition}\label{almostmin}
A function $u\in W^{1,\mathrm{G}}(\Omega)$ is an almost minimizer of $\mathcal{J}_G$ with exponent $\beta>0$ and constant $\kappa\geq0$ if
\begin{enumerate}

\item $u\geq0$ a.e. in $\Omega$;

\item for any ball $B_r(x)$ such that $\overline{B_r(x)}\subset\Omega$ and any $v\in W^{1,G}(B_r(x))$ such that $v=u$ on $\partial B_r(x)$ in the sense of traces (i.e. $v-u\in W_0^{1,G}(B_r(x))$) it holds that
\[
\mathcal{J}_{\mathrm{G}}(u,B_r(x))\leq (1+\kappa r^\beta)\mathcal{J}_{\mathrm{G}}(v,B_r(x)).
\]
\end{enumerate}
\end{definition}

Almost minimizers satisfy the following scaling property:
\begin{lemma}
If $u\in W^{1,G}(B_1)$ is an almost minimizer of $\mathcal{J}_G$ in $B_1$ with exponent $\beta>0$ and constant $\kappa$ and $0<r<1$ then $u_r(x) \defeq \frac{u(rx)}{r}$ is an almost minimizer of $\mathcal{J}_G$ in $B_{1/r}$ with exponent $\beta>0$ and constant $r^\beta\kappa$.
\end{lemma}

\begin{proof}
Assume $u$ is an almost minimizer of $\mathcal{J}_G$ in $B_1$ with exponent $\beta>0$ and constant $\kappa$. Let $\bar{x}\in B_{1/r}$ and $\rho>0$ be such that $\overline{B_\rho(\bar{x})}\subset B_{1/r}$ and $v\in W^{1,G}(B_\rho(\bar{x}))$ such that $v=u_r$ on $\partial B_\rho(\bar{x})$ in the sense of traces.

Denoting $\bar{y}=r\bar{x}$ we have
\[
\overline{B_\rho(\bar{x})}\subset B_{1/r}\:\Rightarrow \:\overline{B_{r\rho}(\bar{y})}\subset B_1,\quad x\in \partial B_{r\rho}(\bar{y})\:\Rightarrow \:\frac{x}{r}\in\partial{B_{\rho}(\bar{x})}.
\]
Then, by hypothesis
\[
rv\left(\frac{x}{r}\right)=ru_r\left(\frac{x}{r}\right)=u(x)
\]
when $x\in \partial B_{r\rho}(\bar{y})$ and therefore
\[
\mathcal{J}_{\mathrm{G}}(u,B_{r\rho}(\bar{y}))\leq (1+\kappa (r\rho)^\beta)\mathcal{J}_{\mathrm{G}}\left(rv\left(\frac{\cdot}{r}\right),B_{r\rho}(\bar{y})\right).
\]
But changing variables we have that
\begin{align*}
\mathcal{J}_{\mathrm{G}}(u,B_{r\rho}(\bar{y})) & = \int_{B_{r\rho}(\bar{y})} \mathrm{G}(|\nabla u(y)|)+\chi_{\{u>0\}}\,dy \\
									& = r^n\int_{B_{\rho}(\bar{x})} \mathrm{G}(|\nabla u(rx)|)+\chi_{\{u>0\}}(rx)\,dx \\
									& = r^n\int_{B_{\rho}(\bar{x})} \mathrm{G}(|\nabla u_r|)+\chi_{\{u_r>0\}}\,dx \\
									& = r^n\mathcal{J}_{\mathrm{G}}\left(u_r,B_{\rho}(\bar{x})\right)
\end{align*}
and a similar computation can be done to show that
\[
\mathcal{J}_{\mathrm{G}}\left(rv\left(\frac{\cdot}{r}\right),B_{r\rho}(\bar{y})\right)=r^n\mathcal{J}_{\mathrm{G}}\left(v,B_{\rho}(\bar{x})\right)
\]
and get the desired inequality.
\end{proof}

\section{Dichotomy lemmas}\label{sec.dichot}

In this section we prove the main technical results needed in the proof of Theorem \ref{thm.main}. We being with the following dichotomy lemma:

\begin{lemma}\label{lem.dichotomy}
There exists $\varepsilon_0\in(0,1)$ such that for every $0<\varepsilon\leq\varepsilon_0$ there exist $\eta\in(0,1)$, $\mathrm{M}>1$ and $\sigma_0\in(0,1)$ depending on $\varepsilon,n,\delta$ and $g_0$ such that: if $\sigma\leq \sigma_0,\:a > \mathrm{M},\:u\in W^{1,\mathrm{G}}(B_1)$ and the following inequality holds for all $v\in W^{1,\mathrm{G}}(B_1)$ such that $v=u$ on $\partial B_1$:
$$
\mathcal{J}_{\mathrm{G}}(u,B_1)\leq (1+\sigma)\mathcal{J}_{\mathrm{G}}(v,B_1)
$$
then, if we denote
\begin{equation}\label{eq.avgu}
\mathrm{G}(a) \defeq \intav{B_1}\mathrm{G}(|\nabla u|)\,dx
\end{equation}
the following dichotomy holds: either
\begin{equation}
\intav{B_\eta}\mathrm{G}(|\nabla u|)\,dx\leq \mathrm{G}\left(\frac{a}{2}\right)
\end{equation}
or
\begin{equation}\label{cota-eg}
\intav{B_\eta}\mathrm{G}(|\nabla u-q|)\,dx\leq \mathrm{G}(\varepsilon a)
\end{equation}
where $q\in\R^n$ verifies
\begin{equation}\label{eq.q}
\mathrm{G}^{-1}\left(\frac{1}{\mathrm{C}}\mathrm{G}\left(\frac{a}{4}\right)\right)<|q|<\mathrm{C}_0 a.
\end{equation}
\end{lemma}

\begin{proof}

Let $v$ be the $g-$harmonic replacement of $u$ in $B_1$. We start by computing, with the aid of Lemma \ref{lem.harmrep} and the definition of almost minimizer,
\begin{align*}
\int_{B_1}\mathrm{G}(|\nabla u-\nabla v|)\,dx&\leq \mathrm{C}\int_{B_1}(\mathrm{G}(|\nabla u|)-\mathrm{G}(|\nabla v|))\,dx\\
&\leq \mathrm{C}\left(\int_{B_1}\mathrm{G}(|\nabla u|)+\chi_{\{u>0\}}\,dx -\mathrm{G}(|\nabla v|)\,dx\right)\\
&=\mathrm{C}\left(\mathcal{J}_{\mathrm{G}}(u,B_1)-\int_{B_1}\mathrm{G}(|\nabla v|)\,dx\right)\\
&\leq \mathrm{C}((1+\sigma)\mathcal{J}_{\mathrm{G}}(v,B_1)-\int_{B_1}\mathrm{G}(|\nabla v|)\,dx)\\
&=\mathrm{C}\left(\sigma\left(\int_{B_1}\mathrm{G}(|\nabla v|)\,dx+|\{v>0\}\cap B_1|\right)+|\{v>0\}\cap B_1|\right)\\
&\leq \mathrm{C}\left(\sigma\int_{B_1}\mathrm{G}(|\nabla v|)\,dx+1\right) \\
&\leq \mathrm{C}\left(\sigma\int_{B_1}\mathrm{G}(|\nabla u|)\,dx+1\right),
\end{align*}
the last inequality owing to the fact that $v$ minimizes the energy. Taking average and recalling \eqref{eq.avgu} we have
\begin{align*}\label{cota1}
\intav{B_1} \mathrm{G}(|\nabla u-\nabla v|)\,dx\leq \mathrm{C}(\sigma \mathrm{G}(a)+1) .
\end{align*}

On the other hand, if we fix $x\in B_{\frac{1}{2}}$, by \eqref{eq.supavg} in Proposition \ref{prop.reg} we have
$$
\mathrm{G}(|\nabla v(x)|)\leq \sup _{B_{\frac{1}{4}}(x)}\mathrm{G}(|\nabla v|)\leq \mathrm{C}\intav{B_{\frac{1}{2}}(x)} \mathrm{G}(|\nabla v|)\, dy\leq \mathrm{C}\intav{B_{1}} \mathrm{G}(|\nabla u|)\, dy\leq \mathrm{C}\mathrm{G}(a)
$$
so that
\begin{equation}\label{cota-gradiente}
\mathrm{G}\left(|\nabla v|\right)\leq C\mathrm{G}(a)\text{ in }B_{1/2}.
\end{equation}
Further, if we denote $q=\nabla v(0)$, and recall \eqref{eq.osc} in Proposition \ref{prop.reg} we have
\begin{align*}
\intav{B_\eta} \mathrm{G}(|\nabla v-q|)\,dx&\leq\intav{B_\eta} \mathrm{G}\left(\mathrm{C}\left(\frac{\eta}{\frac{1}{2}}\right)^{\alpha} \sup_{B_\frac{1}{2}}|\nabla v|\right)\,dx\\
&\leq\left(\mathrm{C}\left(\frac{\eta}{\frac{1}{2}}\right)^{\alpha}\right)^{\delta+1}\intav{B_\eta} \mathrm{G}\left(\sup_{B_\frac{1}{2}}|\nabla v|\right)\,dx\\
&=\mathrm{C}\eta^{\alpha (\delta+1)}\mathrm{G}\left(\sup_{B_\frac{1}{2}}|\nabla v|\right)
\end{align*}
and by \eqref{cota-gradiente}
$$
\intav{B_\eta}\mathrm{G}(|\nabla v-q|)\,dx\leq \mathrm{C}\eta^{\alpha (\delta+1)}\mathrm{G}(a)\qquad\mbox{ for all }\eta\leq\frac{1}{2}.
$$

Now, using the $\Delta_2$ condition (see \eqref{Delta_2cond}),
\begin{align*}
\intav{B_\eta} \mathrm{G}(|\nabla u-q|)\,dx&\leq \mathrm{C}\left(\intav{B_\eta} \mathrm{G}(|\nabla u-\nabla v|)\,dx+\intav{B_\eta} \mathrm{G}(|\nabla v-q|)\,dx\right)\\
&=\mathfrak{I}+\mathfrak{II}
\end{align*}
These two terms are bounded using the previous estimates:
$$
\mathfrak{I}\leq\frac{1}{|B_\eta|}\int_{B_1} \mathrm{G}(|\nabla u-\nabla v|)\,dx\leq\frac{\mathrm{C}}{\eta^n}(\sigma \mathrm{G}(a)+1).
$$
and
$$
\mathfrak{II}\leq \mathrm{C}\eta^{\alpha (\delta+1)} \mathrm{G}(a).
$$
So we get
\begin{equation}\label{cota1.5}
\intav{B_\eta} \mathrm{G}(|\nabla u-q|)\,dx\leq \mathrm{C}\eta^{-n}\sigma \mathrm{G}(a)+ \mathrm{C}\eta^{-n}+\mathrm{C} \mathrm{G}(a)\eta^{\alpha (\delta+1)}
\end{equation}
and
\begin{equation}\label{progradu}
\intav{B_\eta} \mathrm{G}(|\nabla u|)\,dx\leq \mathrm{C}\eta^{-n}\sigma \mathrm{G}(a)+ \mathrm{C}\eta^{-n}+\mathrm{C}\mathrm{G}(a)\eta^{\alpha (\delta+1)}+\mathrm{C}\mathrm{G}(|q|).
\end{equation}
Next, we choose $\sigma=\eta^{n+1}$
$$
\mathrm{C}\eta^{-n}\sigma \mathrm{G}(a)+ \mathrm{C}\eta^{-n}+\mathrm{C}\mathrm{G}(a)\eta^{\alpha(\delta+1)}=\mathrm{C}\eta \mathrm{G}(a)+ \mathrm{C}\eta^{-n}+\mathrm{C}\mathrm{G}(a)\eta^{\alpha (\delta+1)}
$$
and $a$ such that
\begin{equation}\label{cota2}
\mathrm{C}\eta \mathrm{G}(a)+ \mathrm{C}\eta^{-n}+\mathrm{C}\mathrm{G}(a)\eta^{\alpha (\delta+1)}\leq \mathrm{G}(\varepsilon a).
\end{equation}
The inequality \eqref{cota2} holds if
$$
a\geq \mathrm{G}^{-1}\left(\frac{\mathrm{C}\eta^{-n}}{\varepsilon^{g_0+1}-\mathrm{C}\eta-\mathrm{C}\eta^{\alpha (\delta+1)}}\right).
$$
(Note that the denominator is positive if $\eta$ is small enough.) Then,
$$
\intav{B_\eta}\mathrm{G}(|\nabla u|)\,dx\leq \mathrm{G}(\varepsilon a)+\mathrm{C}\mathrm{G}(|q|).
$$

Next we split the two possible cases; if
$$
|q|\leq \mathrm{G}^{-1}\left(\frac{1}{\mathrm{C}}\mathrm{G}\left(\frac{a}{4}\right)\right)
$$
making $\varepsilon<1/4$ we get
$$
\intav{B_\eta} \mathrm{G}(|\nabla u|)\,dx\leq 2\mathrm{G}\left(\frac{a}{4}\right)\leq \frac{2}{2^{\delta+1}} \mathrm{G}\left(\frac{a}{2}\right)\leq \mathrm{G}\left(\frac{a}{2}\right)
$$
which is the first alternative in the Lemma.

On the other hand, if
$$
|q|> \mathrm{G}^{-1}\left(\frac{1}{\mathrm{C}} \mathrm{G}\left(\frac{a}{4}\right)\right)
$$
then by \eqref{cota1.5}
$$
\intav{B_\eta} \mathrm{G}(|\nabla u-q|)\,dx\leq \mathrm{G}(\varepsilon a).
$$
Noting that \eqref{cota-gradiente} implies the bound on $|q|$ we finish the proof.
\end{proof}

\begin{remark}
Using the $\Delta_2$ condition (see in particular \eqref{minmax2}) we can replace \eqref{cota-eg} by
\begin{equation}\label{cota-eg-nueva}
\intav{B_\eta} \mathrm{G}(|\nabla u-q|)\,dx\leq \varepsilon^{\delta+1}\mathrm{G}(a).
\end{equation}
\end{remark}

The following lemma can be seen as an improvement of the previous one in the case where $\varepsilon$ and $\sigma$ are (universally) small enough.

\begin{lemma}\label{lemalargo}
Let $0<a_0<a_1$ and $\mathrm{G}$ a Young function satisfying \eqref{eq.lieberman} such that $\mathrm{G}^{\prime}=g$ is a convex function. Let $u\in W^{1,G}(B_1)$ with $u\geq 0$ and:
$$
\mathcal{J}_{\mathrm{G}}(u,B_1)\leq (1+\sigma)\mathcal{J}_{\mathrm{G}}(v,B_1)
$$
for all $v\in W^{1,\mathrm{G}}(B_1)$ such that $v=u$ on $\partial B_1$. Define $\mathrm{G}(a)$ be as in \eqref{eq.avgu} and assume 	
\[
a_0\leq a\leq a_1.
\]
Assume also that
\[
\intav{B_1} \mathrm{G}(|\nabla u-q|)\,dx\leq \mathrm{G}\left(a\right)\varepsilon^{\delta+1}
\]
for $\varepsilon$ small enough with $q\in\R^n$ satisfying
\begin{equation}\label{eq.q2}
\mathrm{G}^{-1}\left(\frac{1}{\mathrm{C}} \mathrm{G}\left(\frac{a}{8}\right)\right)\leq |q|\leq 2\mathrm{C}_0a
\end{equation}
with $C_0$ given in \eqref{eq.q}.

Then there exist  universal constants $\varepsilon_0,\alpha,r\in(0,1)$ such that if $\varepsilon<\varepsilon_0$ and $\sigma\leq c_0\mathrm{G}(\varepsilon)$ then
$$
\intav{B_r} \mathrm{G}(|\nabla u-\tilde{q}|)\,dx\leq r^{\alpha(\delta+1)} \mathrm{G}(a)\varepsilon^{\delta+1}
$$
where $\tilde{q}\in\R^n$ verifies
$$
|q-\tilde{q}|\leq \tilde{C}\varepsilon^{\tau} a\qquad\mbox{ where }\qquad\tau=\frac{\delta+1}{g_0+1}
$$
\end{lemma}

\begin{proof}
Let $\hat{v}$ be the $g-$harmonic replacement of $u$ in $B_{1/2}$ and define
\[
\begin{cases}
v = \hat{v} &\quad \text{ in } B_{1/2},\\
v = u &\quad \text{ in } B_1\setminus B_{1/2}.
\end{cases}
\]
Since $v\in W^{1,G}(B_1)$ and $v=u$ on $\partial B_1$ we have
\begin{align*}
\mathcal{J}_{\mathrm{G}}(u,B_{1/2}) +\mathcal{J}_{\mathrm{G}}(u,B_1\setminus B_{1/2}) &= \mathcal{J}_{\mathrm{G}}(u,B_1)\\
					&\leq \mathcal{J}_{\mathrm{G}}(v,B_1)+\sigma\mathcal{J}_{\mathrm{G}}(v,B_1) \\
					&= \mathcal{J}_{\mathrm{G}}(v,B_{1/2}) +\mathcal{J}_{\mathrm{G}}(v,B_1\setminus B_{1/2}) +\sigma\mathcal{J}_{\mathrm{G}}(v,B_1).
\end{align*}
so
\[
\mathcal{J}_{\mathrm{G}}(u,B_{1/2})\leq \mathcal{J}_{\mathrm{G}}(v,B_{1/2}) +\sigma\mathcal{J}_{\mathrm{G}}(v,B_1).
\]
Now using the definition of $\mathcal{J}_G$ on the previous inequality we have
\[
\int_{B_{1/2}} \mathrm{G}(|\nabla u|)\,dx+|\{u>0\}\cap B_{1/2}|\leq \int_{B_{1/2}} \mathrm{G}(|\nabla v|)\,dx+|B_{1/2}|+\sigma\mathcal{J}_{\mathrm{G}}(v,B_1)
\]
and hence
\[
\int_{B_{1/2}} \mathrm{G}(|\nabla u|)-\mathrm{G}(|\nabla v|)\,dx\leq |B_{1/2}|-|\{u>0\}\cap B_{1/2}|+\sigma\mathcal{J}_{\mathrm{G}}(v,B_1).
\]
Using Lemma \ref{lem.harmrep}, the fact that
\[
|B_{1/2}|-|\{u>0\}\cap B_{1/2}|=|\{u=0\}\cap B_{1/2}|
\]
and that
\[
\mathcal{J}_{\mathrm{G}}(v,B_1)\leq \int_{B_1} \mathrm{G}(|\nabla u|)\,dx+|B_1|\leq |B_1|\left(\mathrm{G}(a)+1\right)
\]
we arrive at the following inequality:
\begin{equation}\label{eq.desig}
\int_{B_{1/2}} \mathrm{G}(|\nabla u-\nabla v|)\,dx\leq \mathrm{C}\sigma(\mathrm{G}(a)+1)+|\{u=0\}\cap B_{1/2}|.
\end{equation}

Next we will prove that
\begin{equation}\label{eq.boundmeasure}
|\{u=0\}\cap B_{1/2}|\leq \mathrm{C}\varepsilon^{\delta+1}\varepsilon^\gamma
\end{equation}
for some $\gamma\in(0,1)$. We start out by defining
\[
\mathfrak{l}(x) \defeq q\cdot x+b
\]
with
\[
b \defeq \intav{B_1} u(x)\:dx.
\]
Notice that
\[
\intav{B_1} (u(x)-\mathfrak{l}(x))\:dx=-\intav{B_1} q\cdot x\:dx=0.
\]
Denote
\[
\intav{B_1} (u(x)-\mathfrak{l}(x))\:dx=(u-\mathfrak{l})_{B_1}
\]
we have by Poincar\'e inequality
\begin{equation}\label{eq.poinc}
\|u-\mathfrak{l}-(u-\mathfrak{l})_{B_1}\|_{L^{\mathrm{G}}(B_1)}= \|u-\mathfrak{l}\|_{L^{\mathrm{G}}(B_1)} \leq \mathrm{C}\|\nabla (u-\mathfrak{l})\|_{L^{\mathrm{G}}(B_1)}.
\end{equation}

Now, we want to use the fact that
\[
\intav{B_1} \mathrm{G}(|\nabla u-q|)\,dx\leq \mathrm{C} \mathrm{G}\left(a\right)\varepsilon^{\delta+1}
\]
so we need to compare modulars and norms. Applying \eqref{eq.normandmoular} to $w=\nabla (u-l)$ and using the previous inequality we have
\begin{equation}\label{eq.nablaul}
\|\nabla (u-\mathfrak{l})\|_{L^{\mathrm{G}}(B_1)} \leq \mathrm{C}\varepsilon^\tau \mathrm{G}(a)^\frac{1}{g_0+1}.
\end{equation}
and by \eqref{eq.poinc}
\[
\|u-\mathfrak{l}\|_{L^{\mathrm{G}}(B_1)}\leq  \mathrm{C}\varepsilon^\tau \mathrm{G}(a)^\frac{1}{g_0+1}.
\]
Next, we point out that if we denote $\mathfrak{l}^- \defeq \min\{0,l\}$ then $l^-\leq|u-l|$ and hence
\[
\mathrm{G}\left(\frac{\mathfrak{l}^-}{\lambda}\right)\leq \mathrm{G}\left(\frac{|u-\mathfrak{l}|}{\lambda}\right)
\]
and we get
\begin{equation}\label{eq.normaGl}
\|\mathfrak{l}^-\|_{L^{\mathrm{G}}(B_1)}\leq \|u-\mathfrak{l}\|_{L^{\mathrm{G}}(B_1)}\leq  \mathrm{C}\varepsilon^\tau \mathrm{G}(a)^\frac{1}{g_0+1}.
\end{equation}
\normalcolor

Next, we claim that
\begin{equation}\label{eq.lgeqca}
\mathfrak{l}(x)\geq \tilde{\mathrm{c}} \mathrm{G}(a)^\frac{1}{g_0+1}
\end{equation}
in $B_{1/2}$ for sufficiently small $\varepsilon$ and some constant $\tilde{c}$. Let us denote
\[
\sigma \defeq \mathrm{G} (a)^\frac{1}{g_0+1}
\]
to ease notation and suppose that \eqref{eq.lgeqca} does not hold, i.e. for any $c$ there exists $x_0\in B_{1/2}$ such that
\[
\mathfrak{l}(x_0) <  \mathrm{c} \sigma.
\]
Since $\mathfrak{l}(x_0)=q\cdot x_0+b$ and hence
\[
\mathfrak{l}(x)-b\in\left[\frac{-|q|}{2},\frac{|q|}{2}\right]
\]
this yields
\[
\mathrm{c} \sigma>b-\frac{|q|}{2}
\]
or
\begin{equation}\label{eq.c}
\mathrm{c} \sigma+\frac{|q|}{2}>b.
\end{equation}

Let
\[
\mathcal{B}\defeq \left\{x=-\frac{tq}{|q|}+\eta,t\in\left[\frac{6}{8},\frac{7}{8}\right],\eta\in B_{1/8}\right\}
\]
so that, in particular, $\mathcal{B}\subset B_1$.

Now, for $x\in \mathcal{B}$ and using \eqref{eq.c}
\[
\mathfrak{l}(x)=-t|q|+\eta\cdot q+b<\mathrm{c} \sigma+\frac{|q|}{2}-\frac{6}{8}|q|+\frac{1}{8}|q|=\mathrm{c}\sigma-\frac{1}{8}|q|
\]
but then \eqref{eq.q2} gives
\[
\mathfrak{l}(x)<\mathrm{c}\sigma-\frac{1}{8} \mathrm{G}^{-1}\left(\frac{1}{\mathrm{C}} \mathrm{G}\left(\frac{\sigma}{8}\right)\right)
\]
from which we get the bound
\[
\mathfrak{l}(x)<\mathrm{c}\sigma-\frac{\mathrm{C}}{64}\sigma.
\]
which taking $\mathrm{c}>0$ small enough gives $\mathfrak{l}(x)<-\bar{\mathrm{C}}\sigma$ for a positive constant $\bar{\mathrm{C}}$. Inserting this into \eqref{eq.normaGl} gives
\[
\varepsilon^\tau \sigma\geq \|\mathfrak{l}^-\|_{L^{\mathrm{G}}(B_1)}>\|\bar{\mathrm{C}}\sigma\|_{L^{\mathrm{G}}(B_1)}
\]
which is a contradiction if $\varepsilon>0$ is small enough. Therefore \eqref{eq.lgeqca} must hold.

The next step of the proof is divided into two cases according to the embedding properties of Orlicz-Sobolev spaces (see \cite[Sections 8.27-8.35]{adams2003sobolev}). First, if
\[
\int_1^\infty\frac{\mathrm{G}^{-1}(\tau)}{\tau^{1+\frac{1}{n}}}\:d\tau=\infty
\]
we consider the the Sobolev conjugate of $\mathrm{G}$
\[
(\mathrm{G}^\ast)^{-1}(t) \defeq \int_0^t\frac{\mathrm{G}^{-1}(\tau)}{\tau^{1+\frac{1}{n}}}\:d\tau.
\]
Then, we have that
\[
\|u-\mathfrak{l}\|_{L^{\mathrm{G}^\ast}(B_1)}\leq \mathrm{C}\|\nabla (u-\mathfrak{l})\|_{L^{\mathrm{G}}(B_1)}
\]
(here we have also used \eqref{eq.poinc}) which combined with \eqref{eq.nablaul} gives
\[
\|u-\mathfrak{l}\|_{L^{\mathrm{G}^\ast}(B_1)}\leq \mathrm{C}\varepsilon^\tau \mathrm{G}(a)^\frac{1}{g_0+1}.
\]
On the other hand, \eqref{eq.lgeqca} gives that
\[
\|\mathfrak{l}^-\|_{L^{\mathrm{G}^\ast}(B_1)}\geq \mathrm{c} \mathrm{G}(a)^\frac{1}{g_0+1}\|\chi_{\{u=0\}\cap B_{1/2}}\|_{\mathrm{G}^\ast}\geq \mathrm{c}\mathrm{G}(a)^\frac{1}{g_0+1}|\{u=0\}\cap B_{1/2}|^{1/(\delta+1)^\ast}
\]
and putting both inequalities together we obtain
\[
\mathrm{c}\mathrm{G}(a)^\frac{1}{g_0+1}|\{u=0\}\cap B_{1/2}|^{1/(\delta+1)^\ast}\leq \mathrm{C}\varepsilon^\tau \mathrm{G}(a)^\frac{1}{g_0+1}
\]
or
\[
|\{u=0\}\cap B_{1/2}|\leq \mathrm{C}\varepsilon^{\tau(\delta+1)^\ast}=\mathrm{C}\varepsilon^{\delta+1}\varepsilon^{\gamma}
\]
which is \eqref{eq.boundmeasure} with $\gamma=\tau(\delta+1)^\ast-(\delta+1)>0$ (being positive owing to \eqref{eq.exponentes}).

The second case is when
\[
\int_1^\infty\frac{\mathrm{G}^{-1}(\tau)}{\tau^{1+\frac{1}{n}}}\:d\tau<\infty,
\]
in which the Morrey-type embedding holds and (again using Poincar\'e)
\[
\sup_{B_1}|u-\mathfrak{l}|\leq \mathrm{C}\|\nabla (u-\mathfrak{l})\|_{L^{\mathrm{G}}(B_1)}
\]
so that
\[
\sup_{B_1}|u-\mathfrak{l}|\leq \mathrm{C}\varepsilon^\tau \mathrm{G}(a)^\frac{1}{g_0+1}.
\]
But then, using again \eqref{eq.lgeqca}, we have that in $B_{1/2}$
\[
\tilde{\mathrm{c}} \mathrm{G}(a)^\frac{1}{g_0+1}-u\leq \mathfrak{l}-u\leq\sup_{B_1}|u-\mathfrak{l}|\leq \mathrm{C}\varepsilon^\tau \mathrm{G}(a)^\frac{1}{g_0+1}
\]
so that, for $x\in B_{1/2}$
\[
\mathrm{G}(a)^\frac{1}{g_0+1}\left(\tilde{\mathrm{c}}-\mathrm{C}\varepsilon^\tau \right)\leq u(x)
\]
which means that $u(x)>0$ if $\varepsilon>0$ is small enough. In this case
\[
|\{u=0\}\cap B_{1/2}|=0
\]
and \eqref{eq.boundmeasure} is therefore proved.

With the aid of \eqref{eq.boundmeasure}, \eqref{eq.desig} reads
\begin{equation}\label{eq.umenosv}
\int_{B_{1/2}}\mathrm{G}(|\nabla u-\nabla v|)\,dx\leq \mathrm{C}\sigma(\mathrm{G}(a)+1)+\mathrm{C}_0\varepsilon^{\delta+1}\varepsilon^\gamma.
\end{equation}
This readily implies, using \eqref{eq.sums} and the hypothesis, that
\begin{align*}
\int_{B_{1/2}}\mathrm{G}(|\nabla v-q|)\,dx & \leq \mathrm{C}\left(\int_{B_{1/2}}\mathrm{G}(|\nabla u-q|)\,dx +\int_{B_{1/2}} \mathrm{G}(|\nabla u-\nabla v|)\,dx \right)\\
								  & \leq \mathrm{C}\varepsilon^{\delta+1} \mathrm{G}(a)+\mathrm{C}\sigma( \mathrm{G}(a)+1)+\mathrm{C}_0\varepsilon^{\delta+1}\varepsilon^\gamma.			
\end{align*}
Now using the fact that $\sigma\leq \mathrm{c}_0\varepsilon^{\delta+1}$ ($\mathrm{c}_0>0$ to be chosen later) and recalling that $\varepsilon^\gamma\leq \mathrm{G}(a)$ and $a\in[a_0,a_1]$ and the we arrive at
\begin{equation}\label{eq.vmenosq}
\int_{B_{1/2}} \mathrm{G}(|\nabla v-q|)\,dx  \leq \mathrm{C}\varepsilon^{\delta+1} \mathrm{G}(a).			
\end{equation}

The next step is to show that \eqref{eq.vmenosq} implies
\begin{equation}\label{eq.vmenosq'}
|\nabla v(x)-q|\,dx  \leq \mathrm{C}\cdot(\varepsilon a)^\nu,\:\forall x\in B_{1/2}			
\end{equation}
for some $\mathrm{C}>0, \nu\in(0,1)$ and $\varepsilon$ small enough.

Assume that \eqref{eq.vmenosq'} is false, then there exists $\hat{x}\in B_{1/2}$ such that
\[
|\nabla v(\hat{x})-q|\,dx  > \mathrm{C}\cdot(\varepsilon a)^\nu.	
\]
By interior regularity, for any $x\in B_{1/8}(\hat{x})$ we have
\[
|\nabla v(x)-\nabla v(\hat{x})|\,dx  \leq \mathrm{C}|x-\hat{x}|^\alpha
\]
with a constant $\mathrm{C}>0$ that, thanks to \eqref{eq.supavg}, depends only on $n,g_0$ and $a_1$. Then,
\begin{align*}
|\nabla v(x)-q| & \geq |\nabla v(\hat{x})-q|-|\nabla v(x)-\nabla v(\hat{x})| \\
                & > \mathrm{C}\cdot(\varepsilon a)^\nu-\mathrm{C}|x-\hat{x}|^\alpha \\
                & \geq \frac{\mathrm{C}\cdot(\varepsilon a)^\nu}{2}
\end{align*}
as long as $x\in B_{\left(\frac{\mathrm{C}\cdot(\varepsilon a)^\nu}{2}\right)^{1/\alpha}}(\hat{x})$ which is contained in $B_{1/2}$ for $\varepsilon$ small enough. Then
\begin{align*}
\int_{B_{1/2}} \mathrm{G}\left(|\nabla v(x)-q|\right) \:dx & \geq \int_{B_{\left(\frac{\mathrm{C}\cdot(\varepsilon a)^\nu}{2\mathrm{C}}\right)^{1/\alpha}}(\hat{x})} \mathrm{G}\left(|\nabla v(x)-q|\right)\,dx \\
												 & \geq \mathrm{G}\left(\frac{\mathrm{C}\cdot(\varepsilon a)^\nu}{2\mathrm{C}}\right)\left|\frac{\mathrm{C}\cdot(\varepsilon a)^\nu}{2\mathrm{C}}\right|^{n/\alpha}|B_1|.
\end{align*}
This, together with \eqref{eq.vmenosq} gives
\[
\mathrm{G}\left(\frac{\mathrm{C}\cdot(\varepsilon a)^\nu}{2}\right)\left(\frac{\mathrm{C}\cdot(\varepsilon a)^\nu}{2}\right)^{n/\alpha}|B_1|\leq \mathrm{C}\varepsilon^{\delta+1} \mathrm{G}(a)
\]
or, using \eqref{eq.delta2less1},
\[
\left(\frac{\mathrm{C}\cdot(\varepsilon a)^\nu}{2}\right)^{g_0+1+\frac{n}{\alpha}}|B_1|\leq \varepsilon^{\delta+1} \mathrm{G}(a)
\]
which gives a contradiction choosing $\nu$ appropriately and for $\varepsilon$ small enough thus proving \eqref{eq.vmenosq'}.

Next we consider the vectorial function $\overrightarrow{\mathbf{F}}:\R^n\longrightarrow\R^n$ given by
\[
\overrightarrow{\mathbf{F}}(z)\defeq g(|z|)\frac{z}{|z|}
\]
and compute
\begin{align*}
\overrightarrow{\mathbf{F}}(\nabla v)-\overrightarrow{\mathbf{F}}(q) & =\int_0^1\frac{d}{dt}\overrightarrow{\mathbf{F}}\left(q+t(\nabla v-q)\right)\:dt \\
				 & =\int_0^1D\overrightarrow{\mathbf{F}}\left(q+t(\nabla v-q)\right)\cdot(\nabla v-q)\:dt \\
				 & =\mathfrak{A}(x)(\nabla v-q)
\end{align*}
with $\displaystyle \mathfrak{A}(x)\defeq \int_0^1 D\overrightarrow{\mathbf{F}}\left(q+t(\nabla v-q)\right)\:dt$. Nonetheless, we have that
\[
\diver \overrightarrow{\mathbf{F}}(\nabla v)=\diver \overrightarrow{\mathbf{F}}(q)=0
\]
and hence
\[
\diver\left(\mathfrak{A}(x)(\nabla v-q)\right)=0.
\]
Applying interior estimates (i.e. using Proposition \ref{prop.reg} with Remark \ref{remark.reg} and Lemma \ref{lem.ellipticity}) and using \eqref{eq.vmenosq} we have that
\[
\mathrm{G}(|\nabla v(x)-q)|)\leq \sup_{B_{R/4}(x)}\mathrm{G}\left(|\nabla v-q|\right)\leq \mathrm{C}\intav{B_{R/2}(x)}\mathrm{G}\left(|\nabla v-q|\right)\:dx\leq \mathrm{C}\mathrm{G}(a)\varepsilon^{\delta+1}.
\]
or, since $\mathrm{G}$ is increasing,
\[
|\nabla v(x)-q|\leq  \mathrm{C}\mathrm{G}^{-1}\left(\mathrm{G}(a)\varepsilon^{\delta+1}\right).
\]
In particular,
\[
|\bar{q}| \defeq |\nabla v(0)-q|\leq \mathrm{C}\mathrm{G}^{-1}\left(\mathrm{G}(a)\varepsilon^{\delta+1}\right)
\]
so that
\[
|\nabla v(x)-q-\bar{q}|\leq 2\mathrm{C}\mathrm{G}^{-1}\left(\mathrm{G}(a)\varepsilon^{\delta+1}\right),\:x\in B_{1/2}.
\]
This, using again \eqref{eq.osc} in Proposition \ref{prop.reg} we have
\[
\intav{B_r} \mathrm{G}(|\nabla v-q-\bar{q}|)\,dx\leq\intav{B_r} \mathrm{G}\left(\mathrm{C}\left(\frac{r}{\frac{1}{2}}\right)^{\mu} \sup_{B_\frac{1}{2}}|\nabla v-q-\bar{q}|\right)\leq \mathrm{C} r^{\mu(\delta+1)} \mathrm{G}(a)\varepsilon^{\delta+1}.
\]
This and \eqref{eq.umenosv}, together with the doubling condition, give
\begin{align*}
\intav{B_r} \mathrm{G}(|\nabla u-q-\bar{q}|)\,dx & \leq \mathrm{C}\left(\intav{B_r} \mathrm{G}\left(|\nabla u-\nabla v|\right)\:dx +\intav{B_r} \mathrm{G}(|\nabla v-q-\bar{q}|)\,dx \right)\\
									   & \leq \frac{\mathrm{C}}{r^n}\left(\sigma( \mathrm{G}(a)+1)+\mathrm{C}_0\varepsilon^{\delta+1}\varepsilon^\gamma\right)+\mathrm{C}r^{\mu(\delta+1)} \mathrm{G}(a)\varepsilon^{\delta+1}\\
									   & \leq \tilde{\mathrm{C}}r^{-n}\sigma( \mathrm{G}(a)+1)+\tilde{\mathrm{C}}r^{-n}\varepsilon^{\delta+1}\varepsilon^\gamma+\tilde{\mathrm{C}}r^{\mu(\delta+1)} \mathrm{G}(a)\varepsilon^{\delta+1}.
\end{align*}

Now defining $\alpha_0 \defeq \mu$ and for $\alpha<\alpha_0$, we set
\[
r^{\delta+1}=(3\tilde{\mathrm{C}})^{\frac{1}{\alpha-\alpha_0}},\quad \varepsilon_0 \defeq \left(\frac{r^{\alpha(\delta+1)+n} \mathrm{G}(a_0)}{3\tilde{\mathrm{C}}}\right)^{\frac{1}{\gamma}},\quad \mathrm{c}_0 \defeq \frac{r^{\alpha(\delta+1)+n} \mathrm{G}(a_0)}{3\tilde{\mathrm{C}}( \mathrm{G}(a_1)+1)}
\]
so that
\begin{align*}
&\tilde{\mathrm{C}}r^{\mu(\delta+1)} \mathrm{G}(a)\varepsilon^{\delta+1} =\tilde{\mathrm{C}}r^{(\alpha_0-\alpha)(\delta+1)}r^{\alpha(\delta+1)} \mathrm{G}(a)\varepsilon^{\delta+1} =\frac{1}{3}r^{\alpha(\delta+1)} \mathrm{G}(a)\varepsilon^{\delta+1} \\
& \tilde{\mathrm{C}}r^{-n}\varepsilon^{\delta+1}\varepsilon^\gamma \leq \tilde{\mathrm{C}}r^{-n}\frac{r^{\alpha(\delta+1)+n} \mathrm{G}(a_0)}{4\tilde{\mathrm{C}}}\varepsilon^{\delta+1}\leq\frac{1}{3}r^{\alpha(\delta+1)} \mathrm{G}(a)\varepsilon^{\delta+1}  \\
&\tilde{\mathrm{C}}r^{-n}\sigma( \mathrm{G}(a)+1)\leq \tilde{\mathrm{C}}r^{-n}\frac{r^{\alpha(\delta+1)+n} \mathrm{G}(a_0)}{4\tilde{\mathrm{C}}( \mathrm{G}(a_1)+1)}\varepsilon^{\delta+1}(\mathrm{G}(a)+1)\leq \frac{1}{3}r^{\alpha(\delta+1)}\mathrm{G}(a)\varepsilon^{\delta+1}.
\end{align*}

All this gives
\[
\intav{B_r} \mathrm{G}(|\nabla u-q-\bar{q}|)\,dx \leq r^{\alpha(\delta+1)} \mathrm{G}(a)\varepsilon^{\delta+1}
\]
and hence the desired result with $\tilde{q} \defeq q+\bar{q}$.
\end{proof}

As a consequence of the previous lemma we obtain the crucial $L^\infty$ estimate for the gradient of an almost minimizer for sufficiently small $\kappa$.

\begin{corollary}\label{cor}
Let $a_1>a_0>0$ and let $u$ be an almost minimizer of $\mathcal{J}_G$ in $B_1$ with constant $\kappa$ and exponent $\beta$. Set
\[
\mathrm{G}(a) \defeq \intav{B_1} \mathrm{G}(|\nabla u|)\,dx
\]
and assume there exist $q\in\R^n$ and constants $\mathrm{C}_0,\mathrm{C}_1>0$ such that
\begin{equation}\label{eq.hipcor}
\mathrm{G}(a)\in[a_0,a_1],\quad \intav{B_1} \mathrm{G}(|\nabla u-q|)\,dx\leq \mathrm{C}_1\varepsilon^{\delta+1}\mathrm{G}(a),\quad \mathrm{G}^{-1}\left(\frac{1}{C} \mathrm{G}\left(\frac{a}{4}\right)\right)<|q|< 2\mathrm{C}_0 a.
\end{equation}

Then there exist $\varepsilon_0, \kappa_0>0$ and $\gamma\in(0,1)$ depending on $n,\delta,g_0,\beta,a_0$ and $a_1$ such that for $0<\varepsilon\leq \varepsilon_0$ and $0<\kappa\leq \kappa_0\varepsilon^{\delta+1}$ there exists  a linear function
\[
\mathfrak{l}(x)\defeq q\cdot x+b
\]
such that
\[
\|u-\mathfrak{l}\|_{C^{1,\gamma}(B_{1/2})}\leq \mathrm{C}\mathrm{G}^{-1}\left(\mathrm{C} \rho^{-n-\alpha(\delta+1)}\varepsilon^{\delta+1}\mathrm{G}(a)\right)
\]
where $\mathrm{C}=\mathrm{C}(n,\delta,g_0)>0$.

Moreover
\begin{equation}\label{eq.cotagradiente}
\|\nabla u\|_{L^\infty(B_{1/2})}\leq  \hat{\mathrm{C}}a
\end{equation}
where $\hat{\mathrm{C}}>0$ is a universal constant.
\end{corollary}

\begin{proof}

We may assume that $u$ is an almost minimizer in $B_2$ with the same constant and exponent. The goal is to iterate Lemma \ref{lemalargo} with $\alpha\defeq \min\left\{\frac{\alpha_0}{2},\frac{\beta}{g_0+1}\right\}$, $\alpha_0$ from Lemma \ref{lemalargo}.

We are going to show by induction that for some $\rho\in(0,1),\:\tilde{\mathrm{C}}>0$ universal and all $k\geq0$ there exists $q_k\in\R^n$ such that
\begin{align}
&\intav{B_{\rho^k}} \mathrm{G}\left(|\nabla u-q_k|\right)\:dx\leq \mathrm{C}\rho^{k\alpha(\delta+1)}\varepsilon^{\delta+1} \mathrm{G}(a), \label{eq.induccion1}\\
&\mathrm{G}^{-1}\left(\frac{1}{\mathrm{C}} \mathrm{G}\left(\frac{a}{8}\right)\right)-\tilde{\mathrm{C}}\varepsilon^\tau a\left(\frac{1-\rho^{k\alpha\tau}}{1-\rho^{\alpha\tau}}\right)\leq |q_k| \leq 2\mathrm{C}_0a+\tilde{\mathrm{C}}\varepsilon^\tau a\left(\frac{1-\rho^{k\alpha\tau}}{1-\rho^{\alpha\tau}}\right), \label{eq.induccion2} \\
&\mathrm{G}\left(\frac{|q|}{2}\right)\leq \intav{B_{\rho^k}} \mathrm{G}\left(|\nabla u|\right)\:dx\leq 2 \mathrm{G}\left(|q|\right)\label{eq.induccion3}
\end{align}
with
\[
\tau \defeq \left(\frac{\delta+1}{g_0+1}\right).
\]

For $k=0$, pick $q_0\defeq q$; then \eqref{eq.induccion1} and \eqref{eq.induccion2} follow just from the corresponding inequalities in \eqref{eq.hipcor}.

Now for the inductive step let $r\defeq \rho^k$ and $u_r(x) \defeq \frac{1}{r}u(rx)$ and get, by inductive hypothesis, that
\[
\intav{B_1} \mathrm{G}\left(|\nabla u_r-q_k|\right)\:dx\leq \mathrm{C}_1\rho^{k\alpha(\delta+1)}\varepsilon^{\delta+1} \mathrm{G}(a)=\mathrm{C}_1\varepsilon_k^{\delta+1}\mathrm{G}(a)
\]
if we set $\varepsilon_k\defeq \rho^{k\alpha}\varepsilon$. Then, Lemma \ref{lemalargo} implies that there exists $q_{k+1}\in\R^n$ such that
\[
\intav{B_\rho} \mathrm{G}\left(|\nabla u_r-q_{k+1}|\right)\:dx\leq \mathrm{C}_1\varepsilon_k^{\delta+1}\mathrm{G}(a)\rho^{(\delta+1)\alpha}=\mathrm{C}_1\rho^{(k+1)\alpha(\delta+1)}\varepsilon^{\delta+1} \mathrm{G}(a)
\]
and
\begin{equation}\label{eq.difq}
|q_k-q_{k+1}|<\tilde{\mathrm{C}}a\varepsilon_k^\tau=\mathrm{C} a\varepsilon^\tau \rho^{k\alpha\tau}.
\end{equation}
But, rescaling the integral back (and recalling the definition of $r$), we have that
\begin{align*}
\intav{B_\rho} \mathrm{G}\left(|\nabla u_r-q_{k+1}|\right)\:dx & = \frac{1}{\rho^n|B_1|}\int_{B_\rho} \mathrm{G}\left(|\nabla u_r-q_{k+1}|\right)\:dx \\
												  & = \frac{1}{(r\rho)^n|B_1|}\int_{B_r\rho} \mathrm{G}\left(|\nabla u-q_{k+1}|\right)\:dx \\
												  & = \intav{B_{\rho^{k+1}}} \mathrm{G}\left(|\nabla u-q_{k+1}|\right)\:dx
\end{align*}
so \eqref{eq.induccion1} is proved. Incidentally, we point out that the application of Lemma \ref{lemalargo} is where the smallness restriction on $\kappa$ appears.

Regarding \eqref{eq.induccion2}, we can compute, using the inductive hypothesis and \eqref{eq.difq},
\begin{align*}
|q_{k+1}| & \leq |q_{k+1}-q_k|+|q_k| \\
		  & \leq \mathrm{C} a\varepsilon^\tau \rho^{k\alpha\tau}+2\mathrm{C}_0a+\tilde{\mathrm{C}}\varepsilon^\tau a\left(\frac{1-\rho^{k\alpha\tau}}{1-\rho^{\alpha\tau}}\right).
\end{align*}
Then, we have that
\begin{equation}\label{cotaqk}
|q_{k+1}| \leq 2\mathrm{C}_0a+\tilde{\mathrm{C}}\varepsilon^\tau a \left(\rho^{k\alpha\tau} +\frac{1-\rho^{k\alpha\tau}}{1-\rho^{\alpha\tau}}\right)=2\mathrm{C}_0a+\tilde{\mathrm{C}}\varepsilon^\tau a \left(\frac{1-\rho^{(k+1)\alpha\tau}}{1-\rho^{\alpha\tau}}\right)
\end{equation}
which is the upper bound on \eqref{eq.induccion2}. For the lower bound, recall again that the inductive hypothesis and \eqref{eq.difq} give
\begin{align*}
|q_{k+1}| & \geq |q_k| -|q_{k+1}-q_k|\\
 		  &\geq \mathrm{G}^{-1}\left(\frac{1}{\mathrm{C}}\mathrm{G}\left(\frac{a}{8}\right)\right)-\tilde{\mathrm{C}}\varepsilon^\tau a\left(\frac{1-\rho^{k\alpha\tau}}{1-\rho^{\alpha\tau}}\right)-\tilde{\mathrm{C}}a\varepsilon^\tau\rho^{k\alpha\tau}a.
\end{align*}

It remains to prove \eqref{eq.induccion3}.
\begin{align*}
\intav{B_{\rho^{k+1}}}\mathrm{G}\left(|\nabla u|\right)\:dx -\mathrm{G}(|q|) & \leq \frac{\mathrm{C}}{2}\left(\intav{B_{\rho^{k+1}}}\mathrm{G}\left(|\nabla u-q_{k+1}|\right)\:dx +\mathrm{G}(|q_{k+1}|) -\mathrm{G}(|q|)\right)\\
															&\leq \frac{\mathrm{C}}{2}\left(\mathrm{C}\rho^{(k+1)\alpha(\delta+1)}\varepsilon^{\delta+1}\mathrm{G}(a)+\sum_{j=0}^k|\mathrm{G}(|q_{j+1}|) -\mathrm{G}(|q_j|)|\right).
\end{align*}

Now, for some $q_j^\ast$ between $q_{j+1}$ and $q_j$ we have
\[
|\mathrm{G}(|q_{j+1}|) -\mathrm{G}(|q_j|)| =g\left(|q_j^\ast|\right)|q_{j+1}-q_j|
\]
and recalling that $g$ is increasing we have
\begin{align*}
\sum_{j=0}^k|\mathrm{G}(|q_{j+1}|) -\mathrm{G}(|q_j|)| & \leq \sum_{j=0}^k g\left(|q_{j+1}|\right)|q_{j+1}-q_j| \\
&\leq \mathrm{C}\sum_{j=0}^k (\mathrm{c} g(a)+\mathrm{c}\varepsilon^{\delta\tau}g(a)) a(\varepsilon\rho^{j\alpha})^\tau\\
&= \mathrm{C} (\mathrm{c} g(a)a\varepsilon^{\tau}+\mathrm{c}\varepsilon^{\delta\tau}g(a)a\varepsilon^\tau)\sum_{j=0}^k\rho^{j\alpha\tau}\\
&\leq \mathrm{C} \left(\mathrm{G}(a)\varepsilon^{\tau}+\varepsilon^{(\delta+1)\tau}\mathrm{G}(a)\right)\frac{1}{1-\rho^{\alpha\tau}}\\
&=\mathrm{C} \left(1+\frac{1}{1-\rho^{\alpha\tau}}\right)\varepsilon^{\tau}\mathrm{G}(a)
\end{align*}
Putting this inequality together with the previous one, we wind up with
\[
\left|\intav{B_{\rho^{k+1}}}\mathrm{G}\left(|\nabla u|\right)\:dx -\mathrm{G}(|q|) \right|\leq \mathrm{C}(\varepsilon^\tau+\varepsilon^\delta) \mathrm{G}(a)\leq \mathrm{G}\left(\frac{|q|}{2}\right)
\]
as desired.

Now we want to apply the Campanato-type estimate in the Appendix (Theorem \ref{campanato}) with $\Omega=B_{1/2}$ and $\lambda=n+\alpha(\delta+1)>n$; for that, we must first show that
\begin{equation}\label{eq.semicampa}
\varrho^{-\lambda}\inf_{\xi\in\R^n}\int_{B_{1/2}\cap B_{\varrho}(x_0)}\mathrm{G}(|\nabla u-q-\xi|)\:dx\leq \mathrm{C} \varepsilon^{\delta+1}\mathrm{G}(a).
\end{equation}
For that we split the cases $\varrho\in(0,1)$ and $\varrho\geq1$.

In the latter case the bound is obtained by
\[
\varrho^{-\lambda}\inf_{\xi\in\R^n}\int_{B_{1/2}\cap B_{\varrho}(x_0)}\mathrm{G}(|\nabla u-q-\xi|)\:dx \leq \int_{B_{1/2}}\mathrm{G}(|\nabla u-q|)\:dx\leq |B_1|\intav{B_1} \mathrm{G}(|\nabla u-q|)\:dx
\]
and the hypothesis.

If $\varrho\in(0,1)$ then we use a $\rho-$adyc argument: let $\rho^{k+1}\leq\varrho\leq\rho^k$ and use the \eqref{eq.induccion1} to bound
\begin{align*}
\varrho^{-\lambda}\inf_{\xi\in\R^n}\int_{B_{1/2}\cap B_{\varrho}(x_0)} \mathrm{G}(|\nabla u-q-\xi|)\:dx & \leq \rho^{-\lambda(k+1)}\int_{B_{1/2}\cap B_{\rho^k}(x_0)} \mathrm{G}(|\nabla u-q_k|)\:dx \\
																							  & \leq \rho^{-\lambda(k+1)+kn}|B_1|\intav{B_{\rho^k}(x_0)} \mathrm{G}(|\nabla u-q_k|)\:dx \\
																							  & \leq \mathrm{C}\rho^{-(n+\alpha(\delta+1))(k+1)+kn+k\alpha(\delta+1)}|B_1|\varepsilon^{\delta+1} \mathrm{G}(a)\\
																							  & = \mathrm{C}\rho^{-n-\alpha(\delta+1)}|B_1|\varepsilon^{\delta+1} \mathrm{G}(a)
\end{align*}
and we get \eqref{eq.semicampa}.

Since we already have a bound on $\displaystyle \int_{B_{1/2}} \mathrm{G}(|\nabla u-q|)\:dx$, Theorem \ref{campanato} says that $\nabla u-q$ belongs to $C^{0, \gamma}(B_{1/2})$ and
\begin{equation}\label{eq.semiholder}
[\nabla u-q]_{C^{0, \gamma}(B_{1/2})}\leq \mathrm{C}\mathrm{G}^{-1}\left(\rho^{-n-\alpha(\delta+1)}\varepsilon^{\delta+1} \mathrm{G}(a)\right).
\end{equation}

Now if we define $\mathfrak{l}(x)=u(0)+x\cdot q$ then in $B_{1/2}$ it holds
\[
|u(x)-\mathfrak{l}(x)|=\left|\frac{d}{dt}\int_0^1\left(\nabla u(tx)-q\right)\cdot x\:dt\right|\leq \mathrm{C}\mathrm{G}^{-1}\left( \rho^{-n-\alpha(\delta+1)}\varepsilon^{\delta+1} \mathrm{G}(a)\right)
\]
and therefore
\[
\|u-\mathfrak{l}\|_{L^\infty(B_{1/2})}\leq \mathrm{C}\mathrm{G}^{-1}\left( \rho^{-n-\alpha(\delta+1)}\varepsilon^{\delta+1} \mathrm{G}(a)\right)
\]
which together with \eqref{eq.semiholder} and a standard interpolation inequality (see for instance \cite[page 9]{fernandez2022regularity}) gives the first part of the result.

As for \eqref{eq.cotagradiente}, simply note that
\[
\|\nabla u\|_{L^\infty(B_{1/2})}\leq \|\nabla (u-q)\|_{L^\infty(B_{1/2})}+|q|
\]
But taking limit in \eqref{cotaqk} we have
\[
|q|\leq 2\mathrm{C}_0a+\tilde{\mathrm{C}}\varepsilon^\tau a \left(\frac{1}{1-\rho^{\alpha\tau}}\right)
\]
and therefore
\[
\|\nabla u\|_{L^\infty(B_{1/2})}\leq \mathrm{C}\mathrm{G}^{-1}\left(\mathrm{C} \rho^{-n-\alpha(\delta+1)}\varepsilon^{\delta+1} \mathrm{G}(a)\right)+2\mathrm{C}_0a+\tilde{\mathrm{C}}\varepsilon^\tau a \left(\frac{1}{1-\rho^{\alpha\tau}}\right)\leq \hat{\mathrm{C}}a
\]
as desired.
\end{proof}

\section{Proof of the Theorem \ref{thm.main}}\label{sec.main}

In this section we give the proof of Theorem \ref{thm.main}. As is customary we reduce the proof to the case where $\Omega$ is a ball, the general case following from a covering procedure.

\begin{proof}[Proof of Theorem \ref{thm.main}]
Let us define
\[
\mathrm{G}\left(a(\tau)\right) \defeq \intav{B_\tau}\mathrm{G}\left(|\nabla u|\right)\:dx
\]
and for $r\in(0,\eta]$ we want to show that
\begin{equation}\label{eq.boundar}
\mathrm{G}(a(r))\leq \mathrm{C}(\eta, \mathrm{M})\left(1+\mathrm{G}(a(1))\right),
\end{equation}
Let consider the set $\mathcal{K}\subset\N=\{0,1,2,\cdots\}$ containing all the $k$ such that
\begin{equation}\label{eq.recurrence}
\mathrm{G}(a(\eta^k))\leq \eta^{-n}\mathrm{M}+2^{-k}\mathrm{G}(a(1)).
\end{equation}
It is easy to check that $0\in\mathcal{K}$, then $\mathcal{K}\neq\emptyset$.
So, we have two different possibilities: $\mathcal{K}=\N$ or $\mathcal{K}\subset\N$ and $\mathcal{K}\neq\N$.\\

In the case $\mathcal{K}=\N$ we have that for $r\in(0,1)$ there exists a (minimum) value of $k\geq0$ such that $\eta^{k+1}\leq r<\eta^k$. Then
\begin{align*}
\mathrm{G}(a(r)) & = \intav{B_r}\mathrm{G}\left(|\nabla u|\right)\:dx \\
     & \leq \frac{|B_{\eta^k}|}{|B_r|}\intav{B_{\eta^k}} \mathrm{G}\left(|\nabla u|\right)\:dx \\
     & \leq \eta^{-n} \mathrm{G}(a(\eta^{k}))\\
&\leq \eta^{-n}\left( \eta^{-n}\mathrm{M}+2^{-k}\mathrm{G}(a(1))\right)\\
	&\leq \mathrm{C}(\eta,\mathrm{M})(1+\mathrm{G}(a(1)))
\end{align*}
so \eqref{eq.boundar} holds.


In the other case, there exists a minimum integer $k+1$ for which the inequality fails \eqref{eq.recurrence}. Further, it must happen that
\begin{equation}\label{M}
\mathrm{G}(a(\eta^k))>\mathrm{M};
\end{equation}
if not, we have that
\[
\mathrm{G}(a(\eta^{k+1}))\leq \eta^{-n}\mathrm{G}(a(\eta^{k}))\leq \eta^{-n}\mathrm{M}\leq \eta^{-n}\mathrm{M}+2^{-{(k+1)}}\mathrm{G}(a(1))
\]
and \eqref{eq.recurrence} is also satisfied for $k+1$.

Now, $\mathrm{G}(a(\eta^k))>\mathrm{M}$ gives, owing to Lemma \ref{lem.dichotomy} in $B_\eta$ we have that either
\[
\mathrm{G}(a(\eta^{k+1}))\leq \mathrm{G}\left(\frac{a(\eta^k)}{2}\right)
\]
or
$$
\intav{B_{\eta^{k+1}}} \mathrm{G}(|\nabla u-q|)\,dx\leq \varepsilon^{\delta+1}\mathrm{G}(a(\eta^{k}))
$$
where $q\in\R^n$ verifies
$$
\frac{1}{\mathrm{C}}\mathrm{G}\left(\frac{a(\eta^{k})}{4}\right)<\mathrm{G}(|q|)<\mathrm{C} \mathrm{G}(a((\eta^{k})).
$$
Suppose that $ \mathrm{G}(a(\eta^{k+1}))\leq \mathrm{G}(\frac{a(\eta^k)}{2})\leq \frac{\mathrm{G}(a(\eta^k))}{2}$. Since
\begin{equation}\label{mitad}
\mathrm{G}(a(\eta^{k+1}))>\eta^{-n}\mathrm{M}+2^{-{(k+1)}}\mathrm{G}(a(1))\geq\frac{\eta^{-n}\mathrm{M}+2^{-{k}}\mathrm{G}(a(1))}{2}\geq \frac{\mathrm{G}(a(\eta^{k}))}{2}
\end{equation}
this leads to a contradiction.

If the second alternative holds true, then we are in situation to apply Corollary \ref{cor}. Observe that by \eqref{M} and \eqref{mitad} we have that
$$
\mathrm{G}(a(\eta^{k+1}))\geq \frac{\mathrm{M}}{2}.
$$
On the other hand,
$$
\mathrm{G}(a(\eta^{k+1}))\leq\eta^{-n}\mathrm{G}(a(\eta^k))\leq \eta^{-n}(\eta^{-n}\mathrm{M}+ 2^{-k}\mathrm{G}(a(1)))\leq \eta^{-n}(\eta^{-n}\mathrm{M}+ \mathrm{G}(a(1))).
$$
Now, we apply Corollary \ref{cor}, with $a_0=\frac{\mathrm{M}}{2}$ and $a_1=\eta^{-n}(\eta^{-n}\mathrm{M}+ \mathrm{G}(a(1)))$ and taking $2\varepsilon$ instead $\varepsilon$. Then
$$
\intav{B_\eta^{k+1}}\mathrm{G}\left(|\nabla u-q|\right)\:dx\leq (2\varepsilon)^{\delta+1} \mathrm{G}(a(\eta^k))
$$
and
\begin{equation}\label{infinito}
\|\nabla u\|_{L^{\infty}\left( B_{\frac{\eta^{k+1}}{2}}\right)}\leq \mathrm{C}a(\eta^k).
\end{equation}

By \eqref{infinito}, for all $r\in \left(0,\frac{\eta^{k+1}}{2}\right)$.
 \begin{align*}
 \mathrm{G}(a(r))&=\intav{B_r}\mathrm{G}\left(|\nabla u|\right)\:dx\leq \mathrm{C} \mathrm{G}(a(\eta^{k}))\leq \mathrm{C}(\eta^{-n}\mathrm{M}+2^{-k}\mathrm{G}(a(1)))\\
 &\leq \mathrm{C}(\eta^{-n}\mathrm{M}+\mathrm{G}(a(1)))\leq \mathrm{C}(\mathrm{M},\eta)(1+\mathrm{G}(a(1))).
 \end{align*}
If $r\in\left[\frac{\eta^{k+1}}{2},\eta\right)$ then there exists $k_r$ such that $\eta^{k_r+1}<r\leq\eta^{k_r}$. Then
$$
\frac{1}{\eta^{k_r}}\leq\frac{1}{r}\leq\frac{2}{\eta^{k_r+1}}
$$
and
$$
k_r\leq k+\mathrm{C}_{\ast}
$$
So, there are two cases, first suppose that $k_r\in\{0,\cdots,k\}$ then
\begin{align*}
\mathrm{G}(a(r))&=\frac{1}{|B_{\eta^{k_r+1}}|}\int_{B_\eta^{k_r}}\mathrm{G}\left(|\nabla u|\right)\:dx\leq \eta^{-n} \mathrm{G}(\eta^{k_r})\leq \eta^{-n}(\eta^{-n}\mathrm{M}+2^{-k_r}\mathrm{G}(a(1)))\\
 &\leq \eta^{-n}(\eta^{-n}\mathrm{M}+\mathrm{G}(a(1)))\leq \mathrm{C}(\mathrm{M},\eta)(1+\mathrm{G}(a(1))).
\end{align*}
If $k_r>k$ then
\begin{align*}
\mathrm{G}(a(r))&=\frac{1}{|B_{\eta^{k+\mathrm{C}_{\ast}+1}}|}\int_{B_\eta^{k}}\mathrm{G}\left(|\nabla u|\right)\:dx\leq \eta^{-n (\mathrm{C}_{\ast}+1)} \mathrm{G}(\eta^{k})\leq \eta^{-n(\mathrm{C}_{\ast}+1)}(\eta^{-n}\mathrm{M}+2^{-k}\mathrm{G}(a(1)))\\
 &\leq \eta^{-n(\mathrm{C}_{\ast}+1)}(\eta^{-n}\mathrm{M}+\mathrm{G}(a(1)))\leq \mathrm{C}(\mathrm{M},\eta)(1+\mathrm{G}(a(1))).
\end{align*}
so, the proof of \eqref{eq.boundar} is completed.

Now, we extend \eqref{eq.boundar} for balls with center $x_0$ and radius small enough. So, for all $r\in[0,\eta]$ we have
\begin{equation}\label{eq.boundar.general}
\mathrm{G}(a(r,x_0))\leq \mathrm{C}(\eta, \mathrm{M})\left(1+\mathrm{G}(a(1))\right),
\end{equation}
where
\[
\mathrm{G}\left(a(r,x_0)\right)\defeq \intav{B_r(x_0)} \mathrm{G}\left(|\nabla u|\right)\:dx
\]
For $x_0\in B_{\frac{1}{2}}$
$$
\mathrm{G}(|\nabla u(x_0)|)=\lim_{r\to 0} \mathrm{G}(a(r,x_0))\leq \mathrm{C}(\eta, \mathrm{M})(1+\mathrm{G}(a(1))=\mathrm{C}\cdot\left(1+\int_{B_1}\mathrm{G}\left(|\nabla u|\right)\:dx\right)
$$
Finally, we obtain
$$
\|\nabla u\|_{L^{\infty}\left(B_{\frac{1}{2}}\right)}\leq \mathrm{G}^{-1}\left(1+\int_{B_1}\mathrm{G}\left(|\nabla u|\right)\:dx\right)
$$
\end{proof}

\section*{Appendix: a Campanato-type estimate}

In this section we prove a Campanato-type result in the context of Orlicz modulars, i.e. that functions with appropriate average decay are H\"older continuous. As far as the authors' are concerned, this result is new in the specialized literature and hence has independent interest. The proof follows Fern\'andez Bonder's personal notes \cite{bonder} for the homogeneous setting.

Let us denote
\[
\Omega_{x_0,\varrho} \defeq \Omega\cap B_\varrho(x_0)\quad\text{and}\quad u_{x_0,\varrho} \defeq \frac{1}{|\Omega_{x_0,\varrho}|}\int_{\Omega_{x_0,\varrho}}u(x)\:dx = \intav{\Omega_{x_0,\varrho}} u(x)dx.
\]
For our purposes, the analogous of the $L^p$ norm and Campanato seminorm (see for instance \cite{giusti2003direct}) will be given, for $\lambda\geq0$, by
\begin{equation}\label{DefPhi}
  \Phi(u)\defeq \int_\Omega \mathrm{G}(|u|)\:dx,\quad \Phi_\lambda(u) \defeq \sup_{\substack{x_0\in\Omega \\ \varrho>0}}\varrho^{-\lambda}\int_{\Omega_{x_0,\varrho}}\mathrm{G}(|u-u_{x_0,\varrho}|)\:dx
\end{equation}
These quantities are not homogeneous and hence not a norm/seminorm, but they capture the essential properties that will be used to prove the desired continuity.

We will work with functions satisfying
\begin{equation}\label{eq.campanato}
\Phi(u)+\Phi_\lambda(u)\leq \mathrm{C}_{\Phi}
\end{equation}
for some constant $\mathrm{C}_{\Phi}$. The first step is to prove that every point in $\Omega$ is a Lebesgue point of $u$:
\begin{proposition}
If $u$ satisfies \eqref{eq.campanato} and $\lambda>n$ then there exists $\tilde{u}=u$ a.e. such that
\[
\lim_{r\rightarrow0^+}u_{x_0,r}=\tilde{u}(x_0)
\]
uniformly in $\Omega$.
\end{proposition}

\begin{proof}

{\bf Step 1.} Let $0<r_1<r_2$. Then for $y\in B_{r_1}(x_0)$ and using \eqref{eq.sums} we have that
\[
\mathrm{G}(|u_{x_0,r_1}-u_{x_0,r_2}|)\leq 2^{g_0}\left(\mathrm{G}(|u_{x_0,r_1}-u(y)|)+\mathrm{G}(|u_{x_0,r_2}-u(y)|)\right).
\]
Integrating over $B_{r_1}(x_0)$ we get that
\[
|B_1|r_1^n\mathrm{G}(|u_{x_0,r_1}-u_{x_0,r_2}|)\leq 2^{g_0}\Phi_\lambda(u)(r_1^\lambda+r_2^\lambda)
\]
and we get
\[
\mathrm{G}(|u_{x_0,r_1}-u_{x_0,r_2}|)\leq 2^{g_0}\frac{r_1^\lambda+r_2^\lambda}{|B_1|r_1^n}\Phi_\lambda(u).
\]
or
\begin{equation}\label{desig1}
|u_{x_0,r_1}-u_{x_0,r_2}|\leq \mathrm{G}^{-1}\left(2^{g_0}\frac{r_1^\lambda+r_2^\lambda}{|B_1|r_1^n}\Phi_\lambda(u)\right).
\end{equation}

{\bf Step 2.} We will iterate \eqref{desig1} dyadically; for $j\in\mathbb{N}_0$ and fixed $r>0$ we have
\[
\mathrm{G}(|u_{x_0,2^{-(j+1)}r}-u_{x_0,2^{-j}r}|)\leq 2^{g_0}\frac{2^{-(j+1)\lambda}r^\lambda+2^{-j\lambda}r^\lambda}{|B_1|2^{(-j-1)n}r^n}\Phi_\lambda(u).
\]
This means that
\[
|u_{x_0,2^{-(j+1)}r}-u_{x_0,2^{-j}r}|\leq \mathrm{G}^{-1}\left(\frac{2^{g_0}}{|B_1|}\Phi_\lambda(u)\right)r^{\frac{\lambda-n}{g_0+1}}2^{\frac{j(n-\lambda)}{g_0+1}}
\]
and since
\[
|u_{x_0,r}-u_{x_0,2^{-k}r}|\leq \sum_{j=0}^k|u_{x_0,2^{-(j+1)}r}-u_{x_0,2^{-j}r}|
\]
we get
\begin{equation}\label{desig2}
|u_{x_0,r}-u_{x_0,2^{-k}r}|\leq \mathrm{G}^{-1}\left(\frac{2^{g_0}}{|B_1|}\Phi_\lambda(u)\right)r^{\frac{\lambda-n}{g_0+1}}\sum_{j=0}^k 2^{\frac{j(n-\lambda)}{g_0+1}}
\end{equation}

{\bf Step 3.} Now we can use \eqref{desig2} in the following way
\begin{align*}
|u_{x_0,2^{-k}r}-u_{x_0,2^{-(k+i)}r}| & \leq \mathrm{G}^{-1}\left(\frac{2^{g_0}}{|B_1|}\Phi_\lambda(u)\right)\left(\frac{r}{2^k}\right)^{\frac{\lambda-n}{g_0+1}}\sum_{j=0}^{i-1} 2^{\frac{j(n-\lambda)}{g_0+1}} \\
									  & \leq \mathrm{C}\mathrm{G}^{-1}\left(\frac{2^{g_0}}{|B_1|}\Phi_\lambda(u)\right)\left(\frac{r}{2^k}\right)^{\frac{\lambda-n}{g_0+1}}
\end{align*}
where $\mathrm{C}>0$ depends only on $g_0,n$ and $\lambda$. In particular $\left\{u_{x_0,2^{-k}r}\right\}_k$ is a Cauchy sequence (uniformly in $x_0$) and hence has a limit which we call $\tilde{u}(x_0)$:
\[
\lim_{k\rightarrow\infty}u_{x_0,2^{-k}r}=:\tilde{u}(x_0).
\]
Notice that: a) for almost every $x_0$ Lebesgue's Differentiation Theorem says that $\tilde{u}(x_0)=u(x_0)$; and b) $\tilde{u}$ does not depend on the specific choice of $r$. Indeed, passing to the limit as $i\rightarrow\infty$ above we obtain that
\[
|u_{x_0,2^{-k}r}-\tilde{u}(x_0)| \leq \mathrm{C}\mathrm{G}^{-1}\left(\frac{2^{g_0}}{|B_1|}\Phi_\lambda(u)\right)\left(\frac{r}{2^k}\right)^{\frac{\lambda-n}{g_0+1}}
\]
and then, if we consider a $r^{\prime}>0$, we can bound, using \eqref{desig1} and the previous inequality
\begin{align*}
|u_{x_0,2^{-k}r^{\prime}}-\tilde{u}(x_0)| & \leq |u_{x_0,2^{-k}r^{\prime}}-u_{x_0,2^{-k}r}|+|u_{x_0,2^{-k}r}-\tilde{u}(x_0)|  \\
								  & \leq \mathrm{G}^{-1}\left(2^{g_0}\frac{{r^{\prime}}^\lambda+r^\lambda}{|B_1|{r^{\prime}}^n}\Phi_\lambda(u)2^{-k(\lambda-n)}\right)+\mathrm{C}\mathrm{G}^{-1}\left(\frac{2^{g_0}}{|B_1|}\Phi_\lambda(u)\right)\left(\frac{r}{2^k}\right)^{\frac{\lambda-n}{g_0+1}}\\
								  & \leq 2^{-k\frac{\lambda-n}{g_0+1}}\left(\mathrm{G}^{-1}\left(2^{g_0}\frac{{r^{\prime}}^\lambda+r^\lambda}{|B_1|{r^{\prime}}^n}\Phi_\lambda(u)\right)+\mathrm{C}\mathrm{G}^{-1}\left(\frac{2^{g_0}}{|B_1|}\Phi_\lambda(u)\right)r^{\frac{\lambda-n}{g_0+1}}\right)
\end{align*}
which goes to $0$ as $k\rightarrow\infty$.

Finally, we go from discrete to continuous: we pass to the limit on $k$ to obtain in \eqref{desig2} to get
\begin{equation}\label{eq.desig3}
|u_{x_0,r}-\tilde{u}(x_0)|\leq \mathrm{G}^{-1}\left(\frac{2^{g_0}}{|B_1|}\Phi_\lambda(u)\right)r^{\frac{\lambda-n}{g_0+1}}\sum_{j=0}^\infty 2^{\frac{j(n-\lambda)}{g_0+1}}
\end{equation}
and taking the limit $r\rightarrow0^+$ we get the desired uniform convergence.
\end{proof}

Next we want to show that $u$ is indeed H\"older continuous in $\Omega$. We identify it with $\tilde{u}$ and drop the $\:\tilde{}\:$ for simplicity.

\begin{proposition}
If $u$ satisfies \eqref{eq.campanato} and $\lambda>n$ then there exists $\mathrm{C}=\mathrm{C}(n,g_0,\lambda)>0$ such that
\[
\sup_{\substack{x,y\in\Omega \\ x\neq y}}\frac{|u(x)-u(y)|}{|x-y|^\gamma}\leq \mathrm{C}\mathrm{G}^{-1}\left(\frac{2^{g_0}}{|B_1|}\Phi_\lambda(u)\right)
\]
with $\gamma=\frac{\lambda-n}{g_0+1}$.
\end{proposition}

\begin{proof}

We start with the following claim: if $x,y\in\Omega$ and $r=2|x-y|$ then
\[
|u_{x,r}-u_{y,r}|\leq \mathrm{C}\mathrm{G}^{-1}\left(\frac{2^{g_0}}{|B_1|}\Phi_\lambda(u)\right)|x-y|^\gamma
\]
with $\gamma=\frac{\lambda-n}{g_0+1}$.

To prove the claim we note that for $z\in B_{r}(x)\cap B_{r}(y)$ it holds that
\[
\mathrm{G}(|u_{x,r}-u_{y,r}|)\leq 2^{g_0}\left(\mathrm{G}(|u_{x,r}-u(z)|)+\mathrm{G}(|u_{y,r}-u(z)|)\right)
\]
and so integrating over $B_{r}(x)\cap B_{r}(y)$ we get that
\begin{align*}
|B_{r}(x)\cap B_{r}(y)| \mathrm{G}(|u_{x,r}-u_{y,r}|) &\leq 2^{g_0}\left(\int_{B_{r}(x)}\mathrm{G}(|u_{x,r}-u(z)|)\:dz+\int_{B_{r}(y)}\mathrm{G}(|u_{y,r}-u(z)|)\:dz\right) \\
											&\leq 2^{g_0+1}r^\lambda\Phi_\lambda(u)
\end{align*}
which leads to
\[
|u_{x,r}-u_{y,r}| \leq \mathrm{G}^{-1}\left(\frac{1}{|B_{r}(x)\cap B_{r}(y)|}2^{g_0+1}r^\lambda\Phi_\lambda(u)\right).
\]
The claim follows simply by noting that $B_{\frac{r}{2}}(x)\subset B_{r}(x)\cap B_{r}(y)$.

Since
\[
|u(x)-u(y)|\leq |u(x)-u_{x,r}|+|u_{x,r}-u_{y,r}|+|u_{y,r}-u(y)|
\]
the proposition follows using the claim and \eqref{eq.desig3}.
\end{proof}

We remark that $\Phi_\lambda(u)$ is equivalent to
\[
\sup_{\substack{x_0\in\Omega \\ \varrho>0}}\varrho^{-\lambda}\inf_{\xi\in\R}\int_{\Omega_{x_0,\varrho}}\mathrm{G}(|u(x)-\xi|)\:dx;
\]
indeed, one inequality being trivial, the other one follows by noting that Jensen's inequality gives	
\[
\mathrm{G}(|u_{x_0,\varrho}-\xi|)= \mathrm{G}\left(\left|\intav{\Omega_{x_0,\varrho}}( u(x)-\xi)\:dx\right|\right)\leq \intav{\Omega_{x_0,\varrho}} \mathrm{G}\left(\left| u(x)-\xi\right|\right)\:dx
\]	
from where
\begin{align*}
\int_{\Omega_{x_0,\varrho}}\mathrm{G}\left(\left|u(x)- u_{x_0,\varrho}\right|\right)\:dx &\leq 2^{g_0}\left(\int_{\Omega_{x_0,\varrho}}\mathrm{G}\left(\left|u(x)- \xi\right|\right)\:dx+\int_{\Omega_{x_0,\varrho}}\mathrm{G}\left(\left|\xi- u_{x_0,\varrho}\right|\right)\:dx\right)  \\
&\leq 2^{g_0+1}\int_{\Omega_{x_0,\varrho}}\mathrm{G}\left(\left|u(x)- \xi\right|\right)\:dx.
\end{align*}

In summary, we have the following theorem:
\begin{theorem}\label{campanato}
Let $u$ be a measurable function satisfying
\[
\int_\Omega \mathrm{G}(|u|)\:dx+\sup_{\substack{x_0\in\Omega \\ \varrho>0}}\varrho^{-\lambda}\inf_{\xi\in\R}\int_{\Omega_{x_0,\varrho}}\mathrm{G}(|u(x)-\xi|)\:dx\leq \mathrm{C}_0
\]
for some positive constant $\mathrm{C}_0$ and $\lambda>n$. Then, $u\in C^{0, \gamma}(\Omega)$ with $\gamma=\frac{\lambda-n}{g_0+1}$ and
\[
\sup_{\substack{x,y\in\Omega \\ x\neq y}}\frac{|u(x)-u(y)|}{|x-y|^\gamma}\leq \mathrm{C}\mathrm{G}^{-1}\left(\sup_{\substack{x_0\in\Omega \\ \varrho>0}}\varrho^{-\lambda}\inf_{\xi\in\R}\int_{\Omega_{x_0,\varrho}}\mathrm{G}(|u(x)-\xi|)dx\right)
\]
\end{theorem}
\begin{remark}
One could consider $\lambda = n+g_0+1$ in the previous theorem and obtain $u \in C^{0, 1}(\Omega)$, i.e. $u$ is Lipschitz continuous in $\Omega$.
\end{remark}

\begin{definition}[{\cite[Definition 1.2]{SAN21}}] Let $p \in [1,\infty)$ and a $\psi \colon (0,\infty) \to (0,\infty)$ be a continuous function.  We define the generalized Campanato space
$$
\displaystyle \mathcal{L}_{p,\psi}(\Omega) \defeq \left\{f: \Omega \to \mathbb{R}: \,\, \|f\|_{\mathcal{L}_{p,\psi}(\Omega)} \defeq \sup_{B_r(x_0) \subset \Omega } \frac{1}{\psi(r)} \left(\intav{B_r(x_0)} |f(x)-f_{B_r(x_0)}|^pdx\right)^{\frac{1}{p}}< \infty\right\}.
$$
\end{definition}

\begin{remark}
  Observe that $\|f\|_{\mathcal{L}_{p,\psi}(\Omega)}$ is a norm (modulo constant functions), thus $\mathcal{L}_{p,\psi}(\Omega)$ defines a Banach space. Moreover, if $p = 1$ and $\psi \equiv 1$, then $\mathcal{L}_{1,1}(\Omega) = \mathrm{BMO}(\Omega)$. Finally, if $p = 1$ and $\psi(r) = r^{\gamma}$ (for some $0 < \gamma \le 1$), then $\mathcal{L}_{1,r^{\gamma}}(\Omega)$ coincides with standard spaces $C^{0, \gamma}(\Omega)$.
\end{remark}

We also consider now the following functional space
$$
\text{BMO}^{\ast}(\Omega) \defeq \left\{u : \Omega \to \mathbb{R}: \,\,\,\displaystyle \|u\|_{\text{BMO}^{\ast}(\Omega)}  \defeq  \sup_{x_0 \in \Omega \atop{\rho>0} } \intav{\Omega_{x_0, \rho}} |u-u_{x_0, \rho}| dx< \infty \right\}.
$$

Now, we will suppose that $\Omega$ satisfies a $\hat{\mathrm{c}}_0-$uniform positive density ($\hat{\mathrm{c}}_0-$UPD for short), i.e., there exists a universal constant $\hat{\mathrm{c}}_0 \in (0, 1)$ such that
$$
\displaystyle \frac{|\Omega_{x_0, \rho}|}{|B_{\rho}(x_0)|} = \frac{|\Omega \cap B_{\rho}(x_0)|}{|B_{\rho}(x_0)|} \geq \hat{\mathrm{c}}_0 \quad \forall\,\,x_0 \in \overline{\Omega} \,\,\text{and}\,\, \rho>0.
$$
Remember we are assuming that $\Omega$ is a bounded Lipschitz domain, thus it fulfils $\hat{\mathrm{c}}_0-$UPD property for some $\hat{\mathrm{c}}_0 \in (0, 1)$.

Finally, taking into account the above definitions and using Jensen's inequality for concave functions, we can conclude in the borderline case, i.e., $\lambda=n$ (see \eqref{DefPhi} and \eqref{eq.campanato}) the following:
$$
\begin{array}{rcl}
  \displaystyle \|u\|_{\text{BMO}^{\ast}(\Omega)} & \defeq & \displaystyle \sup_{x_0 \in \Omega \atop{\rho>0} } \intav{\Omega_{x_0, \rho}} |u-u_{x_0, \rho}| dx \quad \\
   & = & \displaystyle \sup_{x_0 \in \Omega \atop{\rho>0} } \intav{\Omega_{x_0, \rho}} \mathrm{G}^{-1}\left(\mathrm{G}(|u-u_{x_0, \rho}|)\right) dx \quad \\
   & \le & \displaystyle \sup_{x_0 \in \Omega \atop{\rho>0} }  \mathrm{G}^{-1}\left(\displaystyle \intav{\Omega_{x_0, \rho}} \mathrm{G}(|u-u_{x_0, \rho}|) dx \right) \quad \\
   & \le & \mathrm{G}^{-1}\left(\displaystyle \sup_{x_0 \in \Omega \atop{\rho>0} } \intav{\Omega_{x_0, \rho}} \mathrm{G}(|u-u_{x_0, \rho}|) dx \right) \quad \\
   & \le & \mathrm{G}^{-1}\left(\hat{\mathrm{c}}_0^{-1} \omega_n^{-1}\displaystyle \sup_{x_0 \in \Omega \atop{\rho>0} } \frac{1}{\rho^n}\int_{\Omega_{x_0, \rho}} G(|u-u_{x_0, \rho}|) dx \right) \quad \\
    & \le & \mathrm{G}^{-1}\left(\hat{\mathrm{c}}_0^{-1}\omega_n^{-1}\mathrm{C}_{\Phi}\right).
\end{array}
$$
Therefore, $u \in \text{BMO}^{\ast}(\Omega)$ provided $\lambda =n$ and $\Omega$ fulfils a $\hat{\mathrm{c}}_0-$UPD property.

\section*{Acknowledgments}

The authors would like to thank Juli\'an Fern\'andez Bonder for lending his notes on Morrey estimates for fractional Sobolev spaces, from which the ideas for the Appendix were taken.

A. Silva was partially supported by ANPCyT under grant PICT 2019-3837, by CONICET under grant PIP 11220210100238CO and by UNSL under grant PROICO 03-2023. J.V. da Silva has been partially supported by CNPq-Brazil under Grant No. 307131/2022-0 and FAEPEX-UNICAMP 2441/23 Editais Especiais - PIND - Projetos Individuais (03/2023). H. Vivas was partially supported by ANPCyT under grant PICT 2019-03530.

\bibliographystyle{plain}
\bibliography{biblio}

\end{document}